\documentclass{amsart}

\usepackage{amssymb,latexsym,amscd,amsmath, amsthm, bold-extra, enumitem, dsfont, mathtools}
\usepackage{bm}
\usepackage{color}

\numberwithin{equation}{section}

\long\def\eatit#1{}

\newtheorem{Thm}{Theorem}[section]
\newtheorem{Prop}[Thm]{Proposition}

\theoremstyle{definition}
\newtheorem{Def}[Thm]{Definition}
\newtheorem{Ex}[Thm]{Example}
\newtheorem{Rmk}[Thm]{Remark}

\newcommand{\PP}{{\mathbb{P}}}
\newcommand{\CC}{{\mathbb{C}}}

\begin{document}


\title[Geometry of tensors for modeling rater agreement data]{Algebra and geometry of tensors for modeling rater agreement data}

\author{Cristiano Bocci}
\address{Department of Information Engineering and Mathematics, University of Siena, Via Roma 56, 53100 Siena, Italy}
\email{cristiano.bocci@unisi.it}
\author{Fabio Rapallo}
\address{Department of Science and Technological Innovation, University of Piemonte Orientale, Via Teresa Michel  11, 15121 Alessandria, Italy}
\email{fabio.rapallo@uniupo.it}

\begin{abstract} We study three different quasi-symmetry models and three different mixture models of $n\times n\times n$ tensors for modeling rater agreement data. For these models we give a geometric description of the associated varieties and we study their invariants distinguishing between the case $n=2$ and the case $n>2$. Finally, for the two models for pairwise agreement we state some results about the pairwise Cohen's $\kappa$ coefficients.

\end{abstract}

\maketitle

\section{Introduction}

The problem of rater agreement analysis is a classical topic in Statistics. It can be roughly described as follows. Let us consider $k$ raters which classify $N$ experimental units on the basis of a categorical random variable with $n$ levels. The result of such a process is a $k$-way contingency table which summarizes the counts of the joint rates. From the point of view of Statistics, the problem is to verify if there is agreement among the raters, and to measure the strength of the agreement. For a general introduction to rater agreement analysis, the reader can refer to \cite{von2004analyzing}.

In the case of two raters ($k=2$), there are simple indices to measure the rater agreement, the first one in historical order being the Cohen's $\kappa$, see \cite{cohen1960}. From the Cohen's $\kappa$, several other measures have been introduced and studied in the literature, as for instance the Scott's $\pi$ (which assumes homogeneous marginal distributions of the raters), and the Krippendorff's $\alpha$ originally defined in the framework of content analysis, see \cite{krip1, krip2}. All such measures aims at distinguish the structural agreement from the agreement by chance. We will consider the Cohen's $\kappa$ later in the paper. It is known that the analyses based on these measures of agreement are difficult to explain. For instance, the Cohen's $\kappa$ is equal to $0$ in case of no agreement and is equal to $1$ in case of perfect agreement, but it is hard (and often questionable) to comment about intermediate values of $\kappa$, see the discussion in, e.g., \cite{guggenmoos:93} and the references therein. In this context, even paradoxes are known in literature, see \cite{flight:15}. Moreover, it is problematic to extend such indices to the case of $k>2$ raters. An example in this direction is the Fleiss' $\kappa$, a generalization of the Scott's $\pi$ to more than two raters. A recent extensive review of such indices can be found in \cite{zapf2016}.

To improve the statistical modeling of rater agreement, several directions have been proposed in the literature. A first approach considers the use of log-linear models. For a general reference about log-linear models for contingency tables we refer to \cite{agresti:13, kateri:14}. The first use of log-linear model to model rater agreement dates back to \cite{tanner|young:85}. In particular, we will make use of the quasi-independence model. In the case of two raters, a probability distribution is written as a square $n \times n$ table with generic element $P_{ij}$ (i.e., $P_{ij}$ is the probability of row $i$ and column $j$). In such a case, the quasi-independence model assumes the form:
\begin{equation} \label{qi2}
\log (P_{ij}) = \mu + \lambda^{(1)}_i + \lambda^{(2)}_j + \gamma_i \delta_{(i=j)} \, ,
\end{equation}
where $\mu$ is a global parameter, $\lambda^{(1)}_i$ is the contribution of the $i$-th row, $\lambda^{(2)}_j$ is the contribution of the $j$-th column, and $\gamma_i \delta_{(i=j)}$ is the contribution of the $i$-th diagonal entry (here $\delta$ is the Kronecker symbol). In contrast with the independence log-linear model
\[
\log (P_{ij}) = \mu + \lambda^{(1)}_i + \lambda^{(2)}_j \, ,
\]
it is immediate to see that the two models differ in the expression of the diagonal entries, i.e, the entries of the table where there is agreement among the two raters, and therefore quasi-independence is a good choice to model rater agreement. See \cite{agresti:13} for further details on the quasi-independence model. Other log-linear models commonly used in this field are quasi-symmetry and uniform association models. Another important class of models for this application is the class of mixture models. The connections between quasi-independence and mixture models in the context of rater agreement were studied first in \cite{schuster:02}. We will analyze also mixture models in this paper. When the categorical classification variable is ordinal one can take advantage of this special feature of the classification, and latent class models can be defined, see for instance \cite{nelson:15}. The method proposed there assumes that the $n$ categories of each rater are a discretization of a continuous (Gaussian) underlying classification variable.

In this paper we study the geometry of statistical models for modeling rater agreement in the case of multiple ($k>2$) raters. The varieties associated to these models lie in the space of tensors. In particular, since we mainly focus on the case of $k=3$ raters and categorical random variables with $n$ levels, these varieties lie in $\PP^N$, where $N=n^3-1$. We introduce different models, including models for perfect agreement and models for pairwise agreement. In Section \ref{sgeom} we describe the varieties associated to the  models, introduced in Section \ref{sectionmodels}, in terms of join of varieties \cite{Harris:92} or Hadamard product of varieties \cite{BCK16}. This gives a complete characterization of these varieties from the geometric point of view.
Then, we pass, in Sections \ref{sinv1} and \ref{sinv2}, to study the invariants of the models, i.e. the generators of the ideals of the associated varieties. As expected, the ideals for the quasi-independence models are toric, while the ones for the mixture models are more complicated and, in some situations, it makes the geometric approach not manageable for practical applications. In Section \ref{inclusionsection}, we compare the varieties and we establish some inclusion result. In Section \ref{pairwisekappa} we come back to the statistical problem, and we study how pairwise Cohen's $\kappa$ coefficient can be computed as functions of the model parameters. In the simplest case of equal and uniform marginal distributions, we derive explicit formulae, while in the general case we use Maple${}^\mathrm{TM}$ to do computations. Finally, Section \ref{discussion} is devoted to some final comments.

\section{Models}\label{sectionmodels}

A probability distribution on a finite sample space with $k$ elements is a normalized vector of $k$ non-negative real
numbers. Thus, the most general probability model is the simplex
\begin{equation*} 
\Delta_k = \left\{(P_{1}, \ldots, P_{k}) \ : \ P_{i} \geq 0 \ , \
\sum_{i=1}^k P_{i} = 1   \right\} \, .
\end{equation*}
A statistical model ${\mathcal M}$ is defined to be a subset of
$\Delta_k$. When ${\mathcal M}$ is defined through algebraic equations, the
model ${\mathcal M}$ is said to be an algebraic statistical model.

Given a set of polynomials $F$ with indeterminates $P_1, \ldots,
P_k$, $F$ defines a variety $V(F)$, that is the set of points in
${\mathbb R}^k$ where all polynomials in $F$ vanish. Following
\cite{drton|sturmfels|sullivant:09}, an algebraic statistical
model defined by a list of polynomials $F$, is the set ${\mathcal
M}=V(F) \cap \Delta_k$.

In this paper we consider statistical models for the analysis of rater agreement with 3 raters. All raters use the same nominal rating scale, so that the sample space is naturally a cartesian product of the form $\{1, \ldots, n\}^3$. As a consequence we denote the probabilities $P_{ijk}$ by means of three indices, where $P_{ijk}$ is the probability that an object is rated in the category $i$ by the first rater, in category $j$ by the second rater, and in category $k$ by the third rater.

In this section we describe six different models built over tensors of type $n\times n\times n$.

\begin{Def}\label{Qindef}
The model of quasi-independence $QI_n$ is the set of probability tensors $P\in \Delta_{n\times n \times n}$ such that
\begin{equation}\label{QIeq}
\begin {cases}
P_{ijk}=\zeta\gamma_ia_ib_jc_k \ \mbox{ for } i=j=k\\
P_{ijk}=\zeta a_ib_jc_k \ \mbox{ otherwise }
\end{cases}
\end{equation}
where ${\bm \gamma}, {\mathbf a}, {\mathbf b}, {\mathbf c}$ are non-negative vectors of length  $n$, and $\zeta$ is the normalizing constant.
\end{Def}

\begin{Def}\label{Mixndef}
The mixture model $Mix_n$ is the set of  probability tensors $P\in \Delta_{n\times n \times n}$ such that:
\begin{equation}\label{Mixeq}
\begin{cases}
P_{ijk}=\alpha a_ib_jc_k+(1-\alpha)d_i \ \mbox{ for } i=j=k\\
P_{ijk}=\alpha a_ib_jc_k \ \mbox{ otherwise}.
\end{cases}
\end{equation}
where ${\mathbf a}, {\mathbf b}, {\mathbf c}, {\mathbf d}$ are non-negative vectors of length  $n$ with sum equal to one, and $\alpha \in [0,1]$.
\end{Def}

We can notice that the vector ${\mathbf d}$ can be viewed also as a diagonal probability tensor  $D$  of type $n \times n \times n$.

Imposing the condition $\gamma_1=\cdots =\gamma_n=\gamma$ for the model $QI_n$ we get a new model.

\begin{Def}\label{qindef}
The model  $qI_n$  is the set of  probability tensors $P\in \Delta_{n\times n \times n}$ such that:
\begin{equation}\label{qIeq}
\begin{cases}
P_{ijk}=\zeta \gamma a_i b_j c_k \ \mbox{ for } i=j=k \\
P_{ijk}=\zeta a_i b_j c_k \ \mbox{ otherwise }
\end{cases}
\end{equation}
where ${\mathbf a}, {\mathbf b}, {\mathbf c}$ are non-negative vectors of length  $n$, $\gamma$ is a non-negative number, and $\zeta$ is the normalizing constant.
\end{Def}

The only difference between $QI_n$ and  $qI_n$ lies on the elements in the diagonal: for $qI_n$ there is a contribution given by the same value $\gamma$, while for $QI_n$ the contribution in the $i$-th diagonal entry is given by $\gamma_i$, and two of them $\gamma_i$ and $\gamma_j$ are not necessarily equal. Clearly  $qI_n\subset QI_n$.

Imposing the condition $d_1=\cdots =d_n=d$ for the model $Mix_n$, and thus $d=\frac{1}{n}$, we get a new mixture model.

\begin{Def}\label{mixndef}
The model $mix_n$ is the set of  probability tensors $P\in \Delta_{n\times n \times n}$ such that:
\begin{equation}\label{mixeq}
\begin{cases}
P_{ijk}=\alpha a_i b_j c_k + (1-\alpha)\frac 1 n \ \mbox{ for } i=j=k \\
P_{ijk}=\alpha a_i b_j c_k \ \mbox{ otherwise }
\end{cases}
\end{equation}
where ${\mathbf a}, {\mathbf b}, {\mathbf c}$ are non-negative vectors of length  $n$ with sum equal to one.
\end{Def}
Also in this case one has $mix_n\subset Mix_n$ and the two  models  differ only for the elements in the main diagonal. We can see the non-negative number $d=1/n$ as forming a diagonal probability tensor  $D$ of type $n \times n \times n$.

\begin{Def}\label{pqidef}
The model  of pairwise quasi-independence $p$-$qI_n$ is the set of  probability tensors $P\in \Delta_{n\times n \times n}$ such that
\begin{equation}\label{pqieq}
\begin {cases}
P_{ijk}=\zeta a_ib_jc_k \ \mbox{ if } i,j,k \mbox{ are pairwise distinct}\\
P_{ijk}=\zeta \gamma^{(12)}a_ib_jc_k \ \mbox{ if } i=j\not= k\\
P_{ijk}=\zeta \gamma^{(13)}a_ib_jc_k \ \mbox{ if } i=k\not= j\\
P_{ijk}=\zeta \gamma^{(23)}a_ib_jc_k \ \mbox{ if } i\not=k = j\\
P_{ijk}=\zeta \gamma^{(12)}\gamma^{(13)}\gamma^{(23)}a_ib_jc_k \ \mbox{ if } i=j=k
\end{cases}
\end{equation}
where ${\mathbf a}, {\mathbf b}, {\mathbf c}$ are non-negative vectors of length  $n$, $\gamma^{(12)}, \gamma^{(13)}$ and $\gamma^{(23)}$ are non-negative numbers, and $\zeta$ is the normalizing constant.
\end{Def}

In the definition above, the parameter $\gamma^{(12)}$ is a measure of the agreement between the first and the second rater, and similarly for $\gamma^{(13)}$ and $\gamma^{(23)}$.

\begin{Def}\label{pmixdef}
The pairwise mixture model $p$-$mix_n$ is the set of  probability tensors $P\in \Delta_{n\times n \times n}$ such that:
\begin{equation}\label{pmixeq}
P_{ijk}=\alpha^{(0)} a_ib_jc_k+\alpha^{(12)}\frac{\delta_{ij\bullet}}{n^2}+\alpha^{(13)}\frac{\delta_{i\bullet k}}{n^2}+\alpha^{(23)}\frac{\delta_{\bullet jk}}{n^2}+\alpha^{(123)}\frac{\delta_{ ijk}}{n}
\end{equation}
where ${\mathbf a}, {\mathbf b}, {\mathbf c}$ are non-negative vectors of length  $n$  with sum equal to one, the $\alpha$'s are such that  $(\alpha^{(0)}, \alpha^{(12)}, \alpha^{(13)}, \alpha^{(23)}, \alpha^{(123)}) \in \Delta_5$,  and $\delta$'s are Kronecker symbols.
\end{Def}

\begin{Rmk}
Notice that while in Definitions \ref{Qindef}, \ref{qindef} and \ref{pqidef} the normalization
is applied once, in Definitions \ref{Mixndef}, \ref{mixndef} and \ref{pmixdef} the normalization
is applied twice as we require that each summand must belong to $\Delta_{n\times n \times n}$.
\end{Rmk}

\begin{Rmk}
The models introduced in the Definitions \ref{pqidef} and \ref{pmixdef} are the pairwise extensions of, respectively, the models in Definitions  \ref{qindef} and  \ref{mixndef}.
Similarly, it would be possible to define the pairwise extensions of the models in Definitions  \ref{Qindef} and  \ref{Mixndef}, but they are less common in statistical applications and therefore we omit their study.
\end{Rmk}


\section{Geometry of the models}\label{sgeom}

In this section we describe the varieties associated to the models introduced above.

We work in the projective setting (over $\CC$), using $\PP^{N}$ as ambient space, where $N=n^3-1$ and the homogenous coordinates are $[P_{111}, P_{112}, \dots, P_{ijk}, \dots , P_{nnn}]$ with $1\leq, i,j,k\leq n$. In this way a tensor of type $n\times n\times n$ can be seen as a point in $\PP^N$. If $A\subset \PP^N$, then we denote by $\overline{A}$ the Zariski closure of $A$ in $\PP^N$.

First we briefly recall some definitions.

\begin{Def} The Segre map $s_{n_1,n_2, \dots, n_t}$ is defined as the parametric map
{\small{\[
\begin{array}{cccc}
s_{n_1,n_2, \dots,n_t}:&\PP^{n_1}\times \PP^{n_2}\times \cdots \times \PP^{n_t} & \to& \PP^M\\
&( [a_{1,0}, \dots, a_{1,n_1}],\dots [a_{t,0}, \dots, a_{t,n_t}])& \mapsto & [\dots, a_{1,i_1}a_{2,i_2}\cdots a_{t,i_t}\dots]
\end{array}
\]}}
where $[a_{k,0}, \dots, a_{k,n_k}]$ are the coordinates of $\PP^{n_k}$, for $k=1,\dots, t$ and $M=\prod_{i=1}^t(n_i+1)-1$.

The image of $\PP^{n_1}\times \PP^{n_2}\times \cdots \times \PP^{n_t} $ under $s_{n_1,n_2, \dots, n_t}$ is a projective variety which is called the Segre embedding of $\PP^{n_1}\times \PP^{n_2}\times \cdots \times \PP^{n_t}$ and denoted by $X_{n_1,n_2, \dots, n_t}$.
\end{Def}

In our models there is always a contribution in the parameterization given by $a_ib_jc_k$. This can be seen as the Segre embedding $X_{n-1,n-1,n-1}$ of $\PP^{n-1}\times \PP^{n-1} \times \PP^{n-1}$ in $\PP^N$, where the coordinates are $[a_1, \dots, a_n]$, $[b_1, \dots, b_n]$ and $[c_1, \dots, c_n]$.

\begin{Def}
Let $X$ and $Y$ be two projective varieties in $\PP^n$. The join of $X$ and $Y$, denoted by $J(X,Y)$, is the variety defined as
\[
J(X,Y)=\overline{\{\lambda Q+ \mu R\, : \, Q\in X, R\in Y, \lambda,\mu \in \CC^*\}}
\]
\end{Def}

We can use the previous definition to describe the models $Mix_n$, $mix_n$ and $p$-$mix_n$. Let us start analyzing the model $Mix_n$. Every element here can be seen as a sum of a tensor in the independence model with a diagonal tensor. Hence
\[
V_{Mix_n}=J(X_{n-1,n-1,n-1}, {\mathcal{D}}).
\]
where ${\mathcal{D}}$ is the variety of diagonal tensors, defined as the zero set of the following equations
\[
\{P_{ijk}=0, \mbox{ for } i,j,k \mbox{ not all equal}
\}.
\]

In the case of $mix_n$, every element  is the sum of a tensor in the independence model with the fixed diagonal tensor $D$ whose diagonal entries are $\frac1n$. Hence, by standard  consideration in Algebraic Geometry, $V_{mix_n}$ is the cone with vertex $D$ over the Segre embedding $X_{n-1,n-1,n-1}$.

Finally, for $p$-$mix_n$ let $Q_1$, $Q_2$, $Q_3$ and $Q_4$ be the points in $\PP^N$ representing respectively the tensors $\frac{\delta_{ij\bullet}}{n^2}$, $\frac{\delta_{i\bullet k}}{n^2}$, $\frac{\delta_{\bullet jk}}{n^2}$ and $\frac{\delta_{ijk}}{n}$. Since every element in $p$-$mix_n$ is a linear combination of $Q_1$, $Q_2$, $Q_3$, $Q_4$ and a point in the independence model,  one has
\[
V_{p\mbox{-}mix_n}=J(J(J(J(X_{n-1,n-1,n-1}, Q_1),Q_2),Q_3)Q_4)
\]
that, by a generalization of the argument above, is the cone with vertex the linear span $\langle Q_1,Q_2,Q_3,Q_4\rangle$ over the Segre embedding $X_{n-1,n-1,n-1}$.

To describe the models $QI_n$, $qI_n$ and $p$-$qI_n$ we need to introduce the definition of Hadamard product of varieties. This kind of product was used in \cite{CMS10, CTY10} to describe the algebraic variety associated to the restricted Boltzmann machine, which is the undirected graphical model for binary random variables specified by an appropriate bipartite graph. Recently, in the paper \cite{BCK16}, the authors, motivated by the definition of Hadamard product introduced in \cite{CMS10, CTY10}, start the systematic study of properties of such products, focusing mainly in the case of linear varieties.

\begin{Def} Let $p,q\in \PP^n$ be two points of coordinates respectively $[\alpha_0,\ldots,\alpha_n]$ and $[\beta_0,\ldots,\beta_n]$. If $\alpha_i\beta_i\not= 0$ for some $i$, the Hadamard product $p\star q$ of $p$ and $q$ is defined as
\[
p\star q=[\alpha_0\beta_0,\alpha_1\beta_1,\dots,\alpha_n\beta_n].
\]
If  $\alpha_i\beta_i= 0$ for all $i=0,\dots, n$ then we say $p\star q$ is not defined.
\end{Def}

This definition extends to the Hadamard product of varieties in the following way.
\begin{Def}
Let $X$ and $Y$ be two varieties in $\PP^n$, then their Hadamard product $X\star Y$ is defined as

\[X\star Y=\overline{\{p\star q: p\in X, q\in Y, p\star q\mbox{ is defined} \}}.\]
\end{Def}

Hence, for the model $QI_n$, we can notice that a tensor $P\in QI_n$ can be seen as the Hadamard product $P=S\star H$ where $S$ is a tensor in the independence model and $H$ is a tensor with the $\gamma_i$ on the diagonal and 1 elsewhere. This last kind of tensors are described by the linear variety $\mathcal{H}$ with equations
\[
\{P_{ijk}-P_{lmn}=0\,  \mbox{ for } i,j,k \mbox{ not all equal}  \mbox{ and for } l,m,n \mbox{ not all equal}
\}.
\]
Hence one has $V_{QI_n}=X_{n-1,n-1,n-1}\star \mathcal{H}$.

The cases of $qI_n$ and $p$-$qI_n$ are similar, since
\[
\begin{array}{l}
V_{qI_n}=X_{n-1,n-1,n-1}\star \hat{\mathcal{H}}\\
\\
V_{p\mbox{-}qI_n}=X_{n-1,n-1,n-1}\star \tilde{\mathcal{H}}
\end{array}
\]
where $\hat{\mathcal{H}}$ is the linear variety defined by the equations
\[
\left\{\begin{array}{l}
P_{ijk}-P_{lmn}=0\,  \mbox{ for } i,j,k \mbox{ not all equal}  \mbox{ and for } l,m,n \mbox{ not all equal}\\
\\
P_{111}-P_{iii}=0\, \mbox{ for } i=2,\dots, n
\end{array}
\right\}.
\]
and $\tilde{\mathcal{H}}$ is the linear variety defined by the equations
\[
\left\{\begin{array}{l}
P_{11k}-P_{ttk}=0\, \mbox{ for } t=2,\dots, n \mbox{ and } t\not= k\\
P_{1j1}-P_{tjt}=0\, \mbox{ for } t=2,\dots, n \mbox{ and } t\not= j\\
P_{i11}-P_{itt}=0\, \mbox{ for } t=2,\dots, n \mbox{ and } t\not= i\\
P_{111}-P_{iii}=0\, \mbox{ for } i=2,\dots, n
\end{array}
\right\}.
\]

In the next sections we pass to compute the generators of the ideals of the varieties associated to the models, starting from the case $n=2$.

\section{The case $n=2$}\label{sinv1}

The case $n=2$ can be handled by elimination since the number of variables is small.

For the models $QI_2$ and $Mix_2$ we obtain the following result.

\begin{Prop}\label{QI2n2}
The invariants for the models $QI_2$ and $Mix_2$ are the same, and are a subset of the invariants for the complete independence model, namely
$$P_{121}P_{212}-P_{112}P_{221},$$
$$P_{122}P_{211}-P_{121}P_{212}.$$
\end{Prop}

\begin{Rmk}
$QI_2$ is a toric variety while $Mix_2$ is a subset of $QI_2$. We will discuss this issue later.
\end{Rmk}

Let us turn to $qI_2$ (by definition, we set $\gamma_1 = \gamma_2 = \gamma$ in $QI_2$).
We get the following invariants:
$$P_{121}P_{212}-P_{112}P_{221},$$
$$P_{122}P_{211}-P_{121}P_{212},$$
$$P_{111}P_{212}P_{221}^2-P_{121}P_{211}^2P_{222},$$
$$P_{111}P_{122}P_{221}^2-P_{121}^2P_{211}P_{222},$$
$$P_{111}P_{212}^2P_{221}-P_{112}P_{211}^2P_{222},$$
$$P_{111}P_{122}^2P_{221}-P_{112}P_{121}^2P_{222},$$
$$P_{111}P_{122}P_{212}^2-P_{112}^2P_{211}P_{222},$$
$$P_{111}P_{122}^2P_{212}-P_{112}^2P_{121}P_{222},$$
$$P_{111}P_{112}P_{221}^3-P_{121}^2P_{211}^2P_{222},$$
$$P_{111}P_{122}P_{212}P_{221}-P_{112}P_{121}P_{211}P_{222}.$$

The first two invariants are the ones of the model $QI_2$; moreover we get eight new invariants that we pass to describe.
Let us consider two monomials
$$ f=P_{111}^{\alpha_{111}}P_{112}^{\alpha_{112}}\cdots P_{222}^{\alpha_{222}},$$
$$ g=P_{111}^{\beta_{111}}P_{112}^{\beta_{112}}\cdots P_{222}^{\beta_{222}}.$$
Once we have established the degree $\delta$ of the invariant  we want to find, we must impose that the sum of all the exponents of each monomial be equal to the degree because the relevant toric ideal is clearly homogeneous:
$$\sum_{i=1}^{2}\sum_{j=1}^{2}\sum_{k=1}^{2} \alpha_{ijk}=\sum_{i=1}^{2}\sum_{j=1}^{2}\sum_{k=1}^{2} \beta_{ijk}=\delta.$$

Observing the generators written earlier, we note that if the term $ P_ {111} $ belongs to one monomial then the term $ P_ {222} $ belongs to the other monomial; this happens because $\gamma$ appears only in these two terms. So if we want to find an irreducible binomial $f-g$ with $f$ containing $ P_ {111} $, then $g$ must contain $P_ {222}$ with the same exponent of $ P_ {111} $ in $f$.

Let us consider a $2 \times 3$ matrix $F$, associated with the monomial $f$, defined by
\[
F_{t,s}=\mbox{number of occurrences of } t \mbox{ in the } s\mbox{-th index}
\]
where the occurrences are counted with multiplicity with respect to the exponents $\alpha_{ijk}$ of each variable in $f$. Similarly we have a $2 \times 3$ matrix $G$, associated with the monomial $g$.

\begin{Ex}\label{esempiomatrici1}\rm
Let $\delta=4$ and consider the monomials $f=P_{111}P_{212}P_{221}^2, g=P_{121}P_{211}^2P_{222}$. The associated matrices $F$ e $G$ are:
$$
F = \left(
\begin{array}{ccc}
1 & 2 & 3 \\
3 & 2 & 1
\end{array}
\right)
$$
$$
G = \left(
\begin{array}{ccc}
1 & 2 & 3 \\
3 & 2 & 1
\end{array}
\right)
$$
\end{Ex}

We can state the following result.
\begin{Prop}\label{propmatrici1}
Let $f,g$ be two coprime monomials with associated matrices $F$ and $G$ such that the degree of $P_{111}$ in $f$ is equal to the degree of $P_{222}$ in $g$. Then $F-G=0$ if and only if $ f-g \in I$ where $I \subset K[P_{111},P_{112},\dots ,P_{222}]$ is the ideal associated to the model $qI_2$.
\end{Prop}
\begin{proof}
Suppose $F-G=0$. Hence $\forall t\in\{1,2\}, \forall s\in \{1,2,3\}  \mbox{ } F_{t,s}=G_{t,s}$, which means that the value 1 in the indexes, both in $f$ and $g$, appears the same number of times in the first, second and third position. The same happens for the value 2. By definition of $qI_2$, the number of times that the value 1 appears in first, second and third position determines respectively how many times the terms $a_1$, $b_1$ and $c_1$ appears (in the chosen monomial); similarly for the value 2 and $a_2$, $b_2$ and $c_2$. Since the two matrices are equal, the terms  $a_1,a_2,b_1,b_2,c_1,c_2$ appears in the same way in the monomials of $f-g$. Moreover since in the first monomial there is $P_{111}$ and in the second there is $P_{222}$, $\gamma$ will appear in both monomial of $f-g$. Then we conclude that $f=g$, i.e. $f-g \in I$. The opposite direction just follows by taking an irreducible binomial $f-g$ in $I$ and the matrices $F$ and $G$ and considering the constraints given by (\ref{qIeq}).
\end{proof}

\begin{Ex}\label{esempiomatrici2}\rm
Consider again the monomials of Example \ref{esempiomatrici1}. In this case we get $F=G$ and  $f-g \in I$. As a matter of fact, if we substitute the entries of (\ref{qIeq}) one has
$f=(\zeta\gamma a_1b_1c_1)(\zeta a_2b_1c_2){(\zeta a_2b_2c_1)}^2$ e $g=(\zeta a_1b_2c_1){(\zeta a_2b_1c_1)}^2(\zeta \gamma a_2b_2c_2)$ that is $f=g=\zeta^3 \gamma a_1a_2^3b_1^2b_2^2c_1^3c_2$.
\end{Ex}

In the case of  $mix_2$ we get the following invariants
$$P_{121}P_{212}-P_{112}P_{221},$$
$$P_{122}P_{211}-P_{121}P_{212},$$
$$P_{121}P_{211}^2-P_{111}P_{211}P_{221}-P_{212}P_{221}^2+P_{211}P_{221}P_{222},$$
$$P_{112}P_{211}^2-P_{111}P_{211}P_{212}-P_{212}^2P_{221}+P_{211}P_{212}P_{222},$$
$$P_{121}^2P_{211}-P_{111}P_{121}P_{221}-P_{122}P_{221}^2+P_{121}P_{221}P_{222},$$
$$P_{112}P_{121}P_{211}-P_{111}P_{112}P_{221}-P_{122}P_{212}P_{221}+P_{112}P_{221}P_{222},$$
$$P_{112}^2P_{211}-P_{111}P_{112}P_{212}-P_{122}P_{212}^2+P_{112}P_{212}P_{222},$$
$$P_{112}P_{121}^2-P_{111}P_{121}P_{122}-P_{122}^2P_{221}+P_{121}P_{122}P_{222},$$
$$P_{112}^2P_{121}-P_{111}P_{112}P_{122}-P_{122}^2P_{212}+P_{112}P_{122}P_{222},$$
$$P_{111}P_{112}P_{122}P_{212}+P_{122}^2P_{212}^2-P_{112}^3P_{221}-P_{112}P_{122}P_{212}P_{222}.$$
Also in this case we get 10 generators and the first two are the one of $Mix_2$, while the remaining 8 can be find again with the method of associated matrices, with some modifications. As a matter of fact, we notice that the generators are no more binomials, but tetranomials. Hence we have to choose four monomials
$$f=P_{111}^{\alpha_{111}}P_{112}^{\alpha_{112}}\cdots P_{222}^{\alpha_{222}}$$
$$g=P_{111}^{\beta_{111}}P_{112}^{\beta_{112}}\cdots P_{222}^{\beta_{222}}$$
$$h=P_{111}^{\gamma_{111}}P_{112}^{\gamma_{112}}\cdots P_{222}^{\gamma_{222}}$$
$$w=P_{111}^{\delta_{111}}P_{112}^{\delta_{112}}\cdots P_{222}^{\delta_{222}}$$
and we need to impose the following conditions
\begin{itemize}
\item $h,w$  do not contain  $P_{111}, P_{222}$, that is  $\gamma_{111}=\gamma_{222}=\delta_{111}=\delta_{222}=0$;
\item $f$  must contain $P_{111}$ and  $g$  must contain $ P_{222}$, that is $\alpha_{111}=\beta_{222}=1$;
\item $f,g$ must have the same terms, apart  $P_{111}$ and  $P_{222}$, hence  $\alpha_{222}=\beta_{111}=0$.
\end{itemize}
\begin{Prop}
Let $f,g,h,w$ be the monomials defined as above and let $F,G,H,W$ be their associated matrices. Suppose that $f$ and $h$ are coprime and $g$ and $w$ are coprime. If $H-F=W-G=0$ (or $H-G=W-F=0$) then $h-f+g-w \in I$ where $I \subset K[P_{111},P_{112},\dots ,P_{222}]$ is the ideal associated to $mix_2$.
\end{Prop}
\begin{proof}
The proof  is similar to the  one of Proposition of \ref{propmatrici1}, just apply the procedure twice to pairs of monomial  and check for equivalences.
\end{proof}

A simple parameter count shows that the variety associated to the model $p$-$qI_2$ is a hypersurface in $\PP^7$. Actually it is described by the quartic binomial
\[
P_{1,1,1}P_{1,2,2}P_{2,1,2}P_{2,2,1}-P_{1,1,2}P_{1,2,1}P_{2,1,1}P_{2,2,2}.
\]
On the other hand the model $p$-$mix_2$ is saturated, i.e., the model ideal is the null ideal.

\begin{Rmk}
We saw above that the models $QI_2$ and $qI_2$ are easier than the corresponding models $Mix_2$ and $mix_2$. This remark is still valid for the models of pairwise agreement. Observe that, if we impose $\alpha^{(123)}=0$ in the Equation (\ref{pmixeq}) then the variety associated to this restricted model $p$-$mix_{2|\alpha^{(123)}=0}$  is a hypersurface described by the following sextic polynomial
\[
\begin{array}{c}
x_{1,1,1}^4x_{1,1,2}x_{1,2,1}-
2x_{1,1,1}^2x_{1,1,2}^3x_{1,2,1}+
x_{1,1,2}^5x_{1,2,1}-4x_{1,1,1}^2x_{1,1,2}^2x_{1,2,1}^2+\cdots\\
\\
\cdots + 1304 \mbox{ other terms}+\cdots\\
 \\
 \cdots +x_{1,1,1}x_{2,2,2}^5-x_{1,1,2}x_{2,2,2}^5-x_{1,2,1}x_{2,2,2}^5-x_{2,1,1}x_{2,2,2}^5.
\end{array}
\]

On the other side, every reduced model of the toric model obtained by keeping one or more parameters as fixed is still toric, and therefore described by binomials.

\end{Rmk}


\section{The case $n>2$}\label{sinv2}

Let us start with $QI_n$ and $Mix_n$.
Also in this case, all the invariants  omit diagonal terms.
Moreover there are new invariants which are still binomials, but with a squared term.
A way to construct the invariants for $QI_n$ and $Mix_n$ is the following one: given a degree $\delta$, we choose $d$ non-diagonal terms and we perform  $\delta-1$ changes of indices among the terms in such a way we never get diagonal terms. For example, for $n=3, \delta=2$ one has
$$P_{122}P_{131}-P_{121}P_{132}.$$
The last indices of the two monomials are interchanged. For $\delta=3$ we have, for example:
$$P_{211}P_{233}P_{313}-P_{213}^2P_{331}.$$
In this case we performed two successive changes, the first one on $P_{211}P_{313}$ changing the third index (and obtaining  $P_{213}P_{311}$) and the second one on $P_{233}P_{311}$ changing the second index and obtaining $P_{213}P_{331}$.
In general one has:

\begin{Prop}\label{QIMIXnhigher}
The models $QI_n$ and $Mix_n$ have the same invariants.
\end{Prop}
\begin{proof}
The proof easily follows by looking at the structure of the models. As a matter of fact, the diagonal elements cannot appear in the invariants, since it would not be possible a cancellation of $\gamma_i$'s for the model  $QI_n$ or of $d_i$'s for $Mix_n$. Apart from the diagonal elements, the models are equal (with the exception of a multiplicative constant in $Mix_n$) and the same must be true for their invariants.
\end{proof}

We pass now to study the invariants for the models $qI_n$ and $mix_n$. Here the number of invariants is quite big if compared with the previous models. However, it is still possible to describe them in a combinatorial way. Let us start with some general facts.

Given a monomial  $f=P_{i_1j_1k_1}^{\alpha_{i_1j_1k_1}}\cdots P_{i_mj_mk_m}^{\alpha_{i_mj_mk_m}}$ we associate to it a $t \times 3$ matrix
\[
A_f=\begin{pmatrix}
a_{11} & a_{12} & a_{13}\\
a_{21} & a_{22} & a_{23}\\
\vdots & \vdots & \vdots\\
a_{t1} & a_{t2} & a_{t3}
\end{pmatrix}
\]
with $t=\sum_{s=1}^m \alpha_{i_sj_sk_s}$ and where every row is of the form $(i_s \, j_s \, k_s)$, repeated $\alpha_{i_sj_sk_s}$ times.

Denote by $\rho(A)$ the number of rows with equal entries, i.e.,
\[
\rho(A)=\sharp \{s: a_{s1}=a_{s2}=a_{s3} \}.
\]
Hence, $\rho(A)$ is the number of diagonal elements in the monomial counted with multiplicity.

Let $\sigma_1, \sigma_2$ and $\sigma_3$ permutations of $t$ elements, such that at least two of them are distinct, and define $\sigma=(\sigma_1,\sigma_2,\sigma_3)$. By conditions on $\sigma_1, \sigma_2$ and $\sigma_3$, it follows that $\sigma$ is an element in $(S_t)^3\setminus Diag$, where $S_t$ is the set of permutations of $t$ elements and
\[
Diag=\{(\sigma_1,\sigma_2,\sigma_3)\,:\, \sigma_1=\sigma_2=\sigma_3\}\subset (S_t)^3.
\]
Using  $\sigma_1, \sigma_2$ and $\sigma_3$ we can define a map
\[
F_\sigma: M_{t \times 3}  \to  M_{t \times 3}
\]
sending $A_f$ to
\[
F_\sigma(A_f)=\begin{pmatrix}
a_{\sigma_1(1)1} & a_{\sigma_2(1)2} & a_{\sigma_3(1)3}\\
a_{\sigma_1(2)1} & a_{\sigma_2(2)2} & a_{\sigma_3(2)3}\\
\vdots & \vdots & \vdots\\
a_{\sigma_1(t)1} & a_{\sigma_2(t)2} & a_{\sigma_3(t)3}
\end{pmatrix}.
\]

Let $R_{pq}$ be the $t \times t$ matrix whose entries are $r_{pq}=r_{qp}=1$,  and zero otherwise. The right multiplication of a matrix $M$, of order $t \times l$, by the matrix $R_{pq}$  gives a new $t \times l$ matrix equal to  $M$ except for the $p$-th and $q$-th rows  which are exchanged.

\begin{Def} Let $A$ be a $t \times 3-$matrix. We define two \emph{$S_t$-sets of $A$} as
{\small{\[
\Sigma(A)=\{ F_\sigma(A): \sigma \in (S_t)^3\setminus Diag  \mbox{ with }  \rho(F_\sigma(A)) = \rho(A) \mbox{ and  } \nexists R_{ts} \mbox{ with } R_{ts}\sigma(A)=A \}.
\]
\[
\Sigma'(A)=\{ F_\sigma(A): \sigma \in (S_t)^3\setminus Diag  \mbox{ with }  \rho(F_\sigma(A)) \leq \rho(A) \mbox{ and  } \nexists R_{ts} \mbox{ with } R_{ts}\sigma(A)=A \}.
\]}}
\end{Def}

 Starting from the matrix $A_f$ of a monomial $f$, we compute the sets $\Sigma(A_f)$ and  $\Sigma'(A_f)$, then to every matrix in $\Sigma(A_f)$ (and in $\Sigma'(A_f)$) we can associate, through inverse construction, a new monomial  $P_{i'_1j'_1k'_1}^{\alpha_{i'_1j'_1k'_1}}\cdots P_{i'_mj'_mk'_m}^{\alpha_{i'_mj'_mk'_m}}$.

\begin{Def} Given a monomial $f$ of the form $P_{i_1j_1k_1}^{\alpha_{i_1j_1k_1}}\cdots P_{i_mj_mk_m}^{\alpha_{i_mj_mk_m}}$, denote by $\overline{P_{i_1j_1k_1}^{\alpha_{i_1j_1k_1}}\cdots P_{i_mj_mk_m}^{\alpha_{i_mj_mk_m}}}$ (resp. by $\underline{P_{i_1j_1k_1}^{\alpha_{i_1j_1k_1}}\cdots P_{i_mj_mk_m}^{\alpha_{i_mj_mk_m}}}$) any monomial obtained from a matrix in  $\Sigma(A_f)$ (resp. in $\Sigma'(A_f)$).
\end{Def}

We denote by $S$  the subset of $\{1, 2, \dots, n\}\times\{1, 2, \dots, n\}\times\{1, 2, \dots, n\}$ consisting of all triplets $(i,j,k)$ with at least two distinct entries.

\begin{Thm}\label{thminv1}
The following polynomials are invariants for the model $qI_n$:
\begin{itemize}
\item [(i)] $P_{ijk}P_{i'j'k'}-\overline{P_{ijk}P_{i'j'k'}}=0$ where $(i,j,k), (i',j',k') \in S$ are distinct;
\item [(ii)] $P_{iii}P_{klm}P_{qrs}-\overline{P_{iii}P_{klm}P_{qrs}}=0$ where $(k,l,m),(q,r,s) \in  S$ are distinct and there exists $a\in \{1, \dots, n\}\setminus \{i\}$ such that
\[
a\in \{k,q\}\cap \{l,r\} \cap \{m,s\}
\]
 in such a way changing the indexes we get another diagonal element $P_{aaa}$ (different from $P_{iii}$);
\item [(iii)]$P_{ijk}P_{i'j'k'}P_{i''j''k''}-\overline{P_{ijk}P_{i'j'k'}P_{i''j''k''}}=0$ where $(i,j,k), (i',j',k'), (i'',j'',k'') \in S$ and at least two triplets are distinct;
\item [(iv)] $P_{iii}P_{klm}P_{qrs}P_{xyz}-\overline{P_{iii}P_{klm}P_{qrs}P_{xyz}}=0$ where $(k,l,m),(q,r,s),(x,y,z) \in S$, at least two triplets are distinct and there exists $a\in \{1, \dots, n\}\setminus \{i\}$ such that
\[
a\in \{k,q,x\}\cap \{l,r,y\} \cap \{m,s,z\}
\]
 in such a way changing the indexes we get another diagonal element $P_{aaa}$ (different from $P_{iii}$);
\item [(v)] $P_{iii}P_{klm}P_{qrs}P_{tuv}P_{xyz}-\overline{P_{iii}P_{klm}P_{qrs}P_{tuv}P_{xyz}}=0$  with $(k,l,m)$,$(q,r,s)$, $(t,u,v)$, $(x,y,z) \in S$,at least three triplets are distinct and there exists $a\in \{1, \dots, n\}\setminus \{i\}$ such that
\[
a\in \{k,q,t,x\}\cap \{l,r,u,y\} \cap \{m,s,v,z\}
\]
 in such a way changing the indexes we get another diagonal element $P_{aaa}$ (different from $P_{iii}$);
\end{itemize}
\end{Thm}

\begin{proof}
Case i) are the invariants of the independence model that omit diagonal terms. We prove only the case ii), the others can be proved in a similar way. The monomial $P_{iii}P_{klm}P_{qrs}$ corresponds to $\zeta \gamma a_i b_i c_i a_k b_l c_m a_q b_r c_s$ while the monomial $\overline{P_{iii}P_{klm}P_{qrs}}$ corresponds to a permutation of the indices, according to the positions, giving again $\zeta \gamma a_i b_i c_i a_k b_l c_m a_q b_r c_s$.
\end{proof}

\begin{Thm}\label{thminv2}

The following polynomials are invariants for the model  $mix_n$
\begin{itemize}
\item[(i)] $P_{ijk}P_{i'j'k'}-\underline{P_{ijk}P_{i'j'k'}}=0$ where $(i,j,k), (i',j',k') \in S$ are distinct;
\item[(ii)] $P_{iii}P_{klm}-P_{jjj}P_{klm}+\underline{P_{jjj}P_{klm}}-\underline{P_{iii}P_{klm}}=0$ with $(k,l,m) \in S$ and at least two indices among $k,l$ and $m$ are different from $i$ e $j$;
\item[(iii)] $P_{ijk}P_{qrs}P_{xyz}-\underline{P_{ijk}P_{qrs}P_{xyz}}=0$ where $(i,j,k), (q,r,s), (x,y,z)  \in S$ and at least two triples are distinct;
\item[(iv)] $P_{iii}P_{klm}P_{k'l'm'}-P_{jjj}P_{klm}P_{k'l'm'}-\underline{P_{iii}P_{klm}P_{k'l'm'}}+\underline{P_{jjj}P_{klm}P_{k'l'm'}}=0$ where $(k,l,m),(k',l',m') \in S$ are distinct and at least two indices among $k,l, m$ and among $k',l',m'$ different from $i$ e $j$;
\item[(v)] $P_{iii}P_{jjj}P_{qrs}-P_{iii}P_{kkk}P_{qrs}+P_{jjj}P_{kkk}P_{qrs}-P_{jjj}^2P_{qrs}+P_{jjj}\underline{P_{jjj}P_{qrs}}-P_{iii}\underline{P_{jjj}P_{qrs}}+\underline{P_{iii}P_{kkk}P_{qrs}}-\underline{P_{jjj}P_{kkk}P_{qrs}}$ where $(q,r,s)\in S$, $i,j$ and $k$ are distinct and  $j\notin \{q,r,s\}$.
\item[(vi)] $P_{iii}P_{jjj}^2-P_{iii}^2P_{jjj}+P_{iii}^2P_{kkk}-P_{iii}P_{kkk}^2+P_{jjj}P_{kkk}^2-P_{jjj}^2P_{kkk}+P_{iii}\underline{P_{iii}P_{jjj}}-P_{jjj}\underline{P_{iii}P_{jjj}}+P_{kkk}\underline{P_{iii}P_{kkk}}-P_{iii}\underline{P_{iii}P_{kkk}}+P_{jjj}\underline{P_{jjj}P_{kkk}}-P_{kkk}\underline{P_{jjj}P_{kkk}}$ where $i,j$ and $k$ are distinct;
\item[(vii)] $P_{iii}P_{klm}P_{k'l'm'}P_{k''l''m''}-P_{jjj}P_{klm}P_{k'l'm'}P_{k''l''m''}-\underline{P_{iii}P_{klm}P_{k'l'm'}P_{k''l''m''}}+$ \\$+\underline{P_{jjj}P_{klm}P_{k'l'm'}P_{k''l''m''}}$ where $(k,l,m),(k',l',m'),(k'',l'',m'')\in S$, at least two triplets are distinct and
 at least two indices among $k,l, m$, among $k',l',m'$ and  among $k'',l'',m''$ different from $i$ e $j$;
\end{itemize}
\end{Thm}
\begin{proof}
 The proof is similar to the one of Theorem \ref{thminv1}.
\end{proof}

For the pairwise models $p$-$qI_n$ and $p$-$mix_n$ the number of invariants grows so fast to do not allow a simple description, and this makes the geometric approach not manageable for practical applications.

To underline this fact, we give some computational results for $n=3$. The model $p$-$qI_3$ has $3,436$ invariants which are:
\begin{itemize}
\item[$\bullet$] 52 binomials of degree 3;
\item[$\bullet$] 774 binomials of degree 4;
\item[$\bullet$] 1062 binomials of degree 5;
\item[$\bullet$] 1503 binomials of degree 6;
\item[$\bullet$] 18 binomials of degree 7;
\item[$\bullet$] 27 binomials of degree 8.
\end{itemize}

Instead, the model $p$-$mix_3$ has 359 invariants which are:
\begin{itemize}
\item[$\bullet$] 60 polynomials of degree 2;
\item[$\bullet$]  274 polynomials of degree 3;
\item[$\bullet$]  25 polynomials of degree 4;
\end{itemize}

Even though $p$-$mix_3$ has less invariants then $p$-$qI_3$, we emphasize that the the model $p$-$qI_3$ is described by binomials, while the invariants of $p$-$mix_3$ are not binomials.

\begin{Rmk}
The fact that the models based on quasi-independence are toric has also a computational counterpart. While for mixture models the actual computation of a Gr\"obner basis of the relevant ideal must be performed through elimination, for toric models one can use more specialized algorithms, for instance those implemented in {\tt 4ti2}, see \cite{hemmecke|hemmecke|malkin:05}, leading to a notable reduction in execution time.
\end{Rmk}


\section{Inclusion relations} \label{inclusionsection}

As noticed in the previous sections, there are trivial inclusions between the statistical models defined in Section \ref{sectionmodels}, namely:
\[
qI_n \subset QI_n \qquad mix_n \subset Mix_n \qquad p\textrm{-}mix_n \subset Mix_n \, .
\]

Using the same arguments in \cite{bocci|carlini|rapallo:10} but for three raters, we study in detail the relationship between $QI_n$ and $Mix_n$.

\begin{Prop}\label{inclusione}
In the open simplex $\Delta_{n\times n \times n,>0}$ one has
\[
Mix_n \subset QI_n
\]
\end{Prop}

\begin{proof}
Notice, first of all, that  in $\Delta_{n\times n \times n,>0}$ all parameters in $QI_n$ and $Mix_n$ are strictly positive.
Denote by ${\mathbf a}$, ${\mathbf b}$, ${\mathbf c}$ the parameters in $QI_n$ and by $\overline{{\mathbf a}}$, $\overline{{\mathbf b}}$,$\overline{{\mathbf c}}$ the parameters in $Mix_n$.
Given a generic element in $Mix_n$, this can be written in $QI_n$ by taking
\[
{\mathbf a}=\alpha\overline{{\mathbf a}}, {\mathbf b}=\overline{{\mathbf b}}, {\mathbf c}=\overline{{\mathbf c}},
\]
and
\[
\gamma_i=1+\frac{(1-\alpha)d_i}{\alpha \overline{a}_i\overline{b}_i\overline{c}_i}.
\]
and this concludes the proofs.
\end{proof}

The remaining part of this section is devoted to show that in the boundary of $\Delta_{n\times n \times n}$ no inclusion is verified.

Let us find now a tensor $P\in Mix_n$ such that $P\not\in QI_n$. By choosing $\alpha=0$ in the model $Mix_n$ we have only diagonal tensors (the $D$ ones). Hence if $P\in Mix_n$ one has
$$
\begin{cases}
P_{ijk} = 0  \mbox{ for }  i,j,k  \mbox{ not all equal}\\
P_{iii} = d_i
\end{cases}
$$
Such tensor cannot belong to the model $QI_n$.As a matter of fact, considering, for example,  the entry $P_{112}$ this is zero since it is not in a diagonal position. Recalling the definition of $QI_n$, one has
$$ P_{112}=\zeta a_1b_1c_2.$$
If we take the diagonal entry $P_{111}$, we know that $P_{111}=d_1\neq0$, If we interpret it in $QI_n$ one has
$$P_{111}=\zeta \gamma_1a_1b_1c_1.$$
Since this entry is non-zero, we need to have
$$a_1,b_1,c_1 \neq 0$$
then, from the analysis of $P_{112}$ we get $c_2= 0$.
However, this is not possible since, if we take another diagonal entry, for example $P_{222}$, one has $P_{222}=d_2\neq0$ which, when seen in  $QI_n$ gives
$$P_{222}=\zeta\gamma_2a_2b_2c_2$$
against the fact that $c_2=0$.

However, by virtue of the inclusion in Proposition \ref{inclusione}, there are points belonging to $QI_n\cap  \Delta_{n\times n \times n,>0}$, arbitrarily close to $P$.

Now,  we look for a tensor $P\in QI_n$ such that $P \not\in Mix_n$.
By taking, for example, $\gamma_1 = 0$ we get, in $QI_n$, a tensor $P$ with  $P_{111}=0$ and moreover we can choose all the other parameters in such a way all other entries of $P$ are strictly positive.  Such kind of tensor cannot belong to $Mix_n$.
As a matter of fact, in $QI_n$ one has:
$$
\begin{cases}
P_{ijk} = \zeta a_ib_jc_k  \mbox{ for }  i,j,k  \mbox{ not all equal}\\
P_{ijk}=\zeta \gamma_ia_ib_jc_k \mbox{ for } i=j=k\\
P_{111}=0
\end{cases}
$$
Consider the diagonal element  $P_{111}$. Seeing it in  $Mix_n$ one has
$$P_{111}=\alpha a_1b_1c_1+(1-\alpha)d_1.$$
This must be zero, then
$$\alpha a_1b_1c_1 + (1-\alpha)d_1 = 0.$$
However, this equality forces to have  $P_{ijk}=0$, for some $i,j$ and $k$ with at least one index different from 1, and this contradicts the hypotheses on $P$.

\begin{Rmk}
Notice that model $p$-$qI_n$ does not contain the tensor corresponding to perfect agreement, i.e., the tensor with $1/n$ on the main diagonal and zero otherwise. Nevertheless, this point is on the boundary and can be described as the limit of $P$ in Equation (\ref{pqieq}) when $\gamma^{(12)}=\gamma^{(13)}=\gamma^{(23)} \rightarrow +\infty$.
\end{Rmk}

\section{$p$-$qI_n$, $p$-$mix_n$ and the pairwise Cohen's $\kappa$ coefficients} \label{pairwisekappa}

As mentioned in the Introduction, indices like Cohen's $\kappa$ have been defined to measure the agreement between raters. When there are two raters, the relevant probability distribution is a square matrix $P=(P_{ij}) \in \Delta_{n \times n}$. In such a case, the notion of agreement is clear. The extent of agreement can be read off directly from the table by considering the contributions of the diagonal elements. Large probabilities on the diagonal entries correspond to a strong agreement. The Cohen's $\kappa$ in the two-way setting measures the chance-corrected agreement, i.e., the excess of agreement with respect to a random selection of the categories. In formulae, we have:
\begin{equation} \label{kappa1}
\kappa = \frac {\sum_{i=1}^n P_{ii}- \sum_{i=1}^n P_{i+}P_{+j}} {1-\sum_{i=1}^n P_{i+}P_{+i}}  \, ,
\end{equation}
where $P_{i+}$ is the sum of the $i$-th row and $P_{+j}$ is the sum of the $j$-th column.

In the case of three raters, one can define 3 indices, the pairwise $\kappa$ coefficients. Given a tensor $P = (P_{ijk}) \in \Delta_{n \times n \times n}$, for the first and the second rater one first defines the marginal two-way table $\tilde P^{(12)}$ with entries
\[
\tilde P_{ij} = \sum_{k=1}^n P_{ijk}
\]
and then computes the $\kappa$ coefficient in Equation (\ref{kappa1}) in this table:
\begin{equation}\label{kappa}
\kappa_{12} = \frac {\sum_{i=1}^n \tilde P_{ii}- \sum_{i=1}^n \tilde P_{i+} \tilde P_{+j}} {1-\sum_{i=1}^n \tilde P_{i+} \tilde P_{+i}}  \, .
\end{equation}
The other indices $\kappa_{13}$ and $\kappa_{23}$ are defined in the same way on the other two two-way marginals.

In order to describe the correlations among the pairwise $\kappa$ indices, we consider the pairwise agreement models $p$-$qI_n$ and $p$-$mix_n$ and perform a geometric exploration of the maps
\begin{equation}\label{mappqi}
(\gamma^{(12)},\gamma^{(13)},\gamma^{(23)}) \longmapsto (\kappa_{12},\kappa_{13},\kappa_{23})
\end{equation}
in $p$-$qI_n$ and
\begin{equation} \label{mappmix}
(\alpha^{(0)},\alpha^{(12)},\alpha^{(13)},\alpha^{(23)},\alpha^{(123)}) \longmapsto (\kappa_{12},\kappa_{13},\kappa_{23})
\end{equation}
in $p$-$mix_n$.

First, we consider the special case of equal uniform marginals, i.e. ${\bf a} = {\bf b} = {\bf c} = (\frac 1 n , \ldots, \frac 1 n)$.

\begin{Prop}
If ${\mathbf a}={\mathbf b}={\mathbf c}=\left( \frac1n, \ldots ,\frac1n \right)$ then
\begin{itemize}
\item[i)] for the model $p$-$qI_n$
\[
\kappa_{12}=\frac{\gamma^{(12)}\gamma^{(13)}\gamma^{(23)}+(n-1)\gamma^{(12)}-\gamma^{(13)}-\gamma^{(23)}-n+2}{\gamma^{(12)}\gamma^{(13)}\gamma^{(23)}+(n-1)(\gamma^{(12)}+\gamma^{(13)}+\gamma^{(23)})+(n-1)(n-2)}
\]
and similar formulas hold for $\kappa_{13}$ and $\kappa_{23}$.
\item[ii)] for the model $p$-$mix_n$
\[
\kappa_{12}=\alpha^{(12)}+\alpha^{(123)}
\]
and similar formulas hold for $\kappa_{13}$ and $\kappa_{23}$.
\end{itemize}
\end{Prop}
\begin{proof}

Since ${\mathbf a}={\mathbf b}={\mathbf c}=\left( \frac1n, \ldots \frac1n \right)$, the tensor $P\in p$-$qI_n$ has
\begin{itemize}
\item $n$ entries equal to $ \frac{\zeta}{n^3}\gamma^{(12)}\gamma^{(13)}\gamma^{(23)}$;
\item $n(n-1)$ entries equal to $\frac{\zeta}{n^3}\gamma^{(12)}$;
\item $n(n-1)$ entries equal to $\frac{\zeta}{n^3}\gamma^{(13)}$;
\item $n(n-1)$ entries equal to $\frac{\zeta}{n^3}\gamma^{(23)}$;
\item and $n^3-3n^2+2n$ entries equal to $\frac{\zeta}{n^3}$.
\end{itemize}
Since the sum of all entries of $P$ is equal to 1 we get:
\begin{equation}\label{somma}
1=\zeta\left[\frac{1}{n^2}\gamma^{(12)}\gamma^{(13)}\gamma^{(23)}+\frac{n-1}{n^2}\left(\gamma^{(12)}+\gamma^{(13)}+\gamma^{(23)}\right)+\frac{(n-1)(n-2)}{n^2}\right]
\end{equation}
from which we recover that
\[
\zeta=\frac{n^2}{T(\gamma^{(12)},\gamma^{(13)},\gamma^{(23)})}
\]
where $T(\gamma^{(12)},\gamma^{(13)},\gamma^{(23)})$ is the polynomial in the $\gamma$'s
\[
T(\gamma^{(12)},\gamma^{(13)},\gamma^{(23)})=\gamma^{(12)}\gamma^{(13)}\gamma^{(23)}+(n-1)(\gamma^{(12)}+\gamma^{(13)}+\gamma^{(23)})+(n-1)(n-2).
\]

Let us analyze now the terms $\tilde P_{ij}$ in the marginal two-way table $\tilde P^{(12)}$. If $i=j$, then $\tilde P_{ij}=\sum_{k=1}^n P_{ijk}$ consists of $n-1$ terms where the first and second indices are equal and one term where all indices are equal, hence one has
\begin{equation}\label{Pii}
\tilde P_{ii}=\zeta\left[\frac{(n-1)}{n^3}\gamma^{(12)}+\frac{1}{n^3}\gamma^{(12)}\gamma^{(13)}\gamma^{(23)}\right].
\end{equation}
If $i\not=j$, then $\tilde P_{ij}=\sum_{k=1}^n P_{ijk}$ consists of $n-2$ terms where the three indices are distinct, one term where $i=k$ and one term where $j=k$, hence one has
\begin{equation}\label{Pij}
\tilde P_{ij}=\zeta\left[\frac{(n-2)}{n^3}+\frac{1}{n^3}\gamma^{(13)}+\frac{1}{n^3}\gamma^{(13)}\right].
\end{equation}
From (\ref{Pii}) we get
\begin{equation}\label{diag}
\begin{split}
\sum_{i=1}^n \tilde P_{ii}&=n\zeta \left[\frac{(n-1)}{n^3}\gamma^{(12)}+\frac{1}{n^3}\gamma^{(12)}\gamma^{(13)}\gamma^{(23)}\right]=\\
&=\frac{\zeta}{n^2}\left((n-1)\gamma^{(12)}+\gamma^{(12)}\gamma^{(13)}\gamma^{(23)}\right).
\end{split}
\end{equation}
Notice that $\tilde P_{i+}$ consists of the sum of one term with equal indices, $\tilde P_{ii}$, and $n-1$ terms with distinct indices $\tilde P_{ij}$ and the same happens for $\tilde P_{+i}$, hence one has
\begin{equation*}
\begin{split}
\tilde P_{+i}=\tilde P_{i+}&=\zeta\left[\frac{(n-1)}{n^3}\gamma^{(12)}+\frac{1}{n^3}\gamma^{(12)}\gamma^{(13)}\gamma^{(23)}\right]+\\
&+\zeta\left[(n-1)\left(\frac{(n-2)}{n^3}+\frac{1}{n^3}\gamma^{(13)}+\frac{1}{n^3}\gamma^{(13)} \right)\right]\\
&=\frac{\zeta}{n^3}\cdot\left(\gamma^{(12)}\gamma^{(13)}\gamma^{(23)}+(n-1)(\gamma^{(12)}+\gamma^{(13)}+\gamma^{(23)}) +(n-1)(n-2)\right)\\
&=\frac{\zeta}{n^3}T(\gamma^{(12)},\gamma^{(13)},\gamma^{(23)})=\frac{1}{n}.
 \end{split}
\end{equation*}
Hence
\begin{equation}\label{rc}
\sum_{i=1}^n \tilde P_{i+} \tilde P_{+j}=\frac{1}{n}.
\end{equation}
Substituting (\ref{somma}), (\ref{diag}) and (\ref{rc})  in (\ref{kappa}) we get
\begin{equation*}
\begin{split}
\kappa_{12}& = \frac {\sum_{i=1}^n \tilde P_{ii}- \sum_{i=1}^n \tilde P_{i+} \tilde P_{+j}} {1-\sum_{i=1}^n \tilde P_{i+} \tilde P_{+i}}=\frac{\frac{\zeta}{n^2}\left((n-1)\gamma^{(12)}+\gamma^{(12)}\gamma^{(13)}\gamma^{(23)}\right)-\frac{1}{n}}{1-\frac{1}{n}}=\\
&=\left(\frac{1}{T(\gamma^{(12)},\gamma^{(13)},\gamma^{(23)})}\left((n-1)\gamma^{(12)}+\gamma^{(12)}\gamma^{(13)}\gamma^{(23)}\right)-\frac{1}{n} \right)\cdot \frac{n}{n-1}=\\
&=\frac{n\left[(n-1)\gamma^{(12)}+\gamma^{(12)}\gamma^{(13)}\gamma^{(23)}\right]-T(\gamma^{(12)},\gamma^{(13)},\gamma^{(23)})}{(n-1)\cdot T(\gamma^{(12)},\gamma^{(13)},\gamma^{(23)})}=\\
&=\frac{\gamma^{(12)}\gamma^{(13)}\gamma^{(23)}+(n-1)\gamma^{(12)}-\gamma^{(13)}-\gamma^{(23)}-n+2}{\gamma^{(12)}\gamma^{(13)}\gamma^{(23)}+(n-1)(\gamma^{(12)}+\gamma^{(13)}+\gamma^{(23)})+(n-1)(n-2)}.
\end{split}
\end{equation*}

Let us turn to case  $ii)$. Again,
since ${\mathbf a}={\mathbf b}={\mathbf c}=\left( \frac1n,\ldots \frac1n \right)$, the tensor $P\in p$-$mix_n$ has
\begin{itemize}
\item $n$ entries equal to $\frac{\alpha^{(0)}}{n^3}+\frac{\alpha^{(12)}}{n^2}+\frac{\alpha^{(13)}}{n^2}+\frac{\alpha^{(23)}}{n^2}+\frac{\alpha^{(123)}}{n}$
\item $n(n-1)$ entries equal to $\frac{\alpha^{(0)}}{n^3}+\frac{\alpha^{(12)}}{n^2}$;
\item $n(n-1)$ entries equal to $\frac{\alpha^{(0)}}{n^3}+\frac{\alpha^{(13)}}{n^2}$;
\item $n(n-1)$ entries equal to $\frac{\alpha^{(0)}}{n^3}+\frac{\alpha^{(23)}}{n^2}$;
\item and $n^3-3n^2+2n$ entries equal to $\frac{\alpha^{(0)}}{n^3}$.
\end{itemize}
Reasoning as in case $i)$ one has
\[
\begin{array}{l}
\sum_{i=1}^n \tilde P_{ii}=\frac{\alpha^{(0)}+n\alpha^{(12)}+\alpha^{(13)}+\alpha^{(23)}+n\alpha^{(123)}}{n}=\frac{1+(n-1)\alpha^{(12)}+(n-1)\alpha^{(123)}}{n}\\
\\
\tilde P_{+i}=\tilde P_{i+}=\frac{1}{n}\\
\\
\sum_{i=1}^n \tilde P_{i+} \tilde P_{+j}=n\frac{1}{n^2}=\frac{1}{n}
\end{array}
\]
Hence
\begin{equation*}
\begin{split}
\kappa_{12}& = \frac {\sum_{i=1}^n \tilde P_{ii}- \sum_{i=1}^n \tilde P_{i+} \tilde P_{+j}} {1-\sum_{i=1}^n \tilde P_{i+} \tilde P_{+i}} =\\
&=\frac{\frac{1+(n-1)\alpha^{(12)}+(n-1)\alpha^{(123)}}{n}-\frac{1}{n}}{1-\frac{1}{n}}=\frac{\frac{(n-1)(\alpha^{(12)}+\alpha^{(123)})}{n}}{\frac{n-1}{n}}=\\
&=\alpha^{(12)}+\alpha^{(123)}.
\end{split}
\end{equation*}
\end{proof}

In the general case, we use the software Maple${}^\mathrm{TM}$ \cite{maple} to study the image of the maps in Equations (\ref{mappqi}) and (\ref{mappmix}) for fixed values of $\bf a, \bf b, \bf c$. Notice that when the $\gamma$ parameters are greater than or equal to $0$ in the $p$-$qI_n$ model, then the pairwise $\kappa$ indices are in $[0,1]$ and the same happens for arbitrary values of $\alpha^{(0)},\alpha^{(12)},\alpha^{(13)},\alpha^{(23)},\alpha^{(123)}$ with sum one in the $p$-$mix_n$ model. Therefore, we consider the images of the maps in the cube $[0,1]^3$. The Maple${}^\mathrm{TM}$ code to obtain the figures presented here is reported in the Appendix. All figures are based on the grids $\{10^0, 10^{0.1}, 10^{0.2}, \ldots, 10^{2}\}^3$ for the $\gamma$'s and $\{0, 0.1, \ldots, 1\}^3$ for the $\alpha$'s.

\begin{figure}
\begin{tabular}{cc}
\includegraphics[width=6cm]{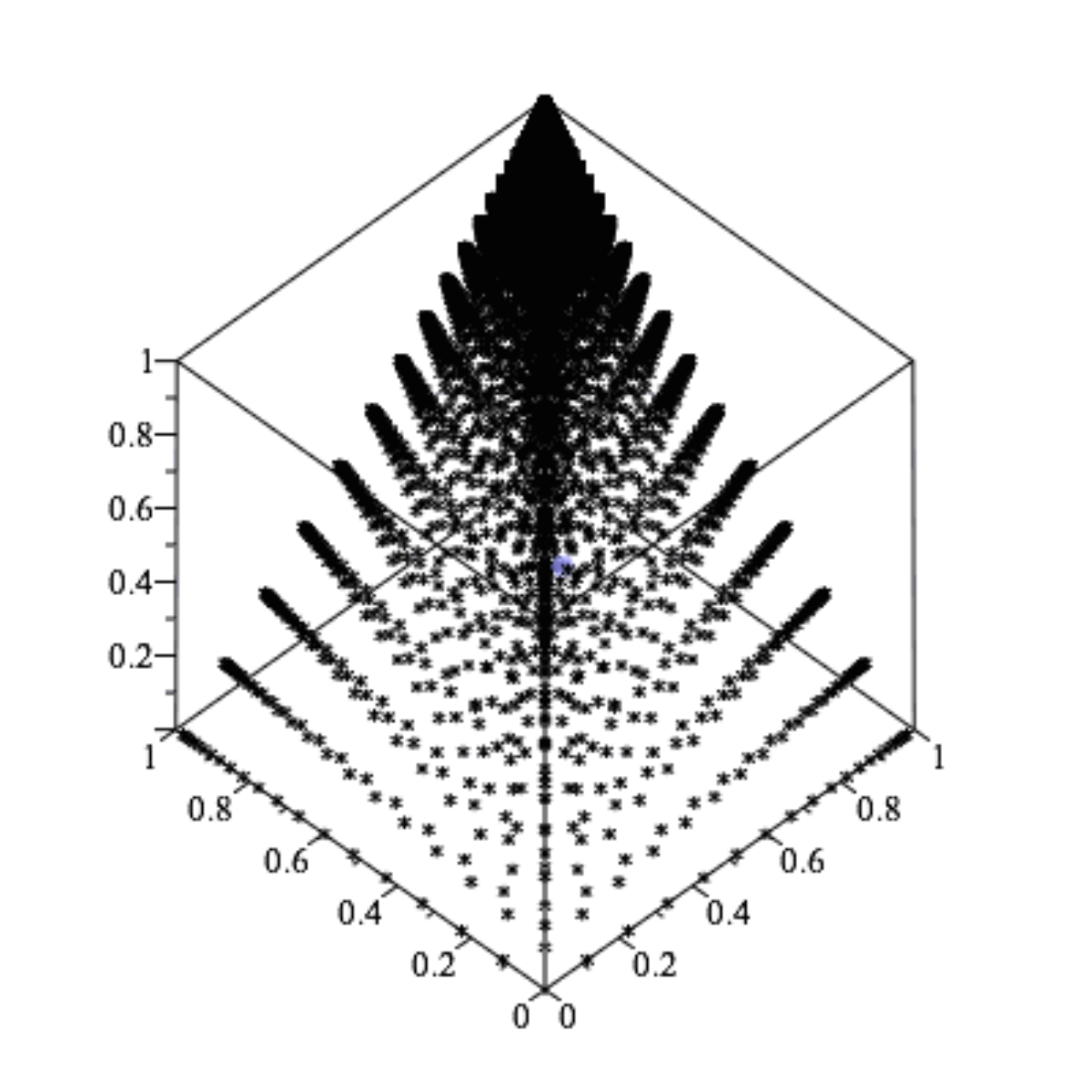} &\includegraphics[width=6cm]{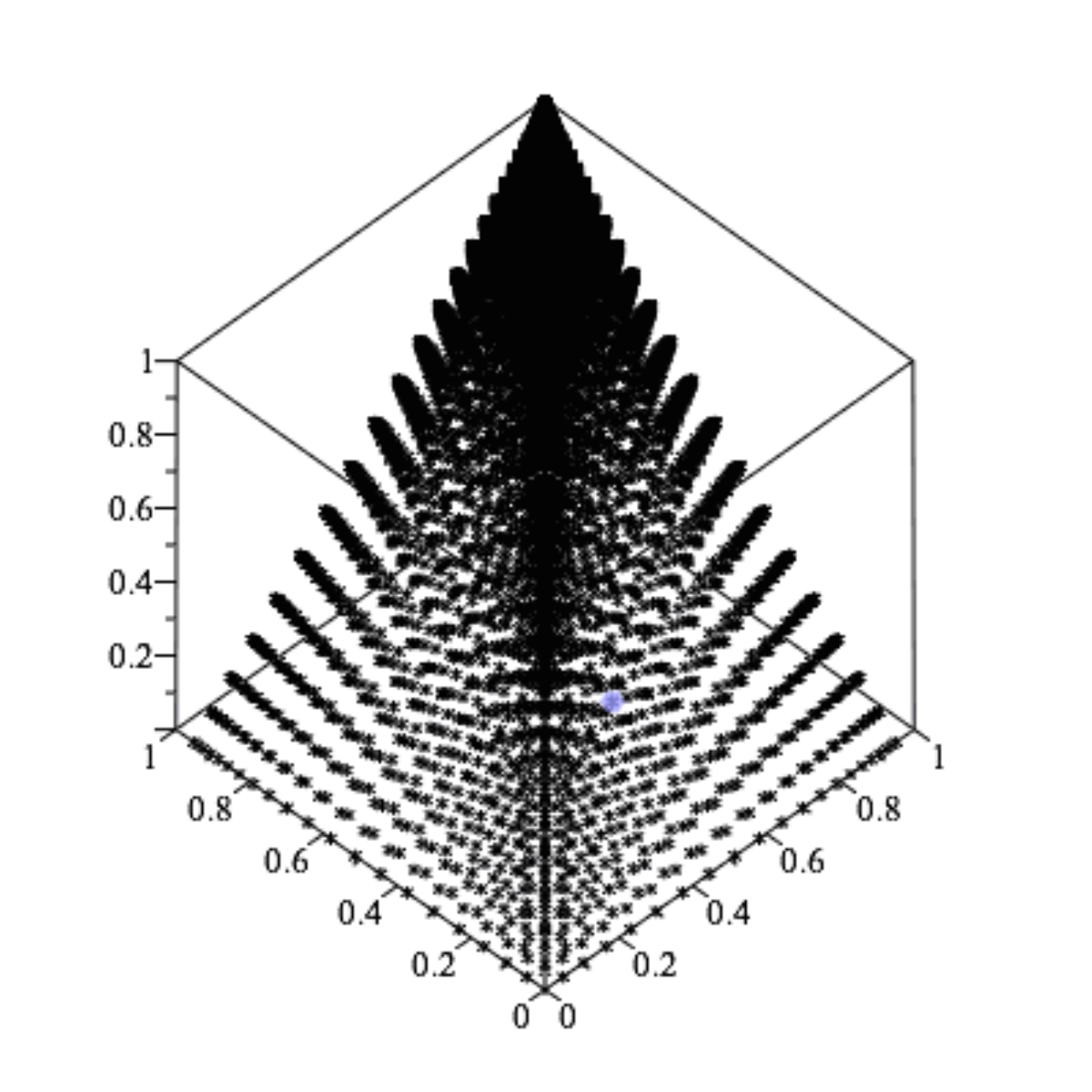}
\\
${\mathbf a}={\mathbf b}={\mathbf c}=\left( \frac12, \frac12\right)$&  ${\mathbf a}={\mathbf b}={\mathbf c}=\left( \frac15, \frac15, \frac15, \frac15, \frac15\right)$\\
\\
$\kappa_{12}=\frac{\gamma^{(12)}\gamma^{(13)}\gamma^{(23)}+\gamma^{(12)}-\gamma^{(13)}-\gamma^{(23)}}{\gamma^{(12)}\gamma^{(13)}\gamma^{(23)}+\gamma^{(12)}+\gamma^{(13)}+\gamma^{(23)}}$ &
$\kappa_{12}=\frac{\gamma^{(12)}\gamma^{(13)}\gamma^{(23)}+4\gamma^{(12)}-\gamma^{(13)}-\gamma^{(23)}-3}{\gamma^{(12)}\gamma^{(13)}\gamma^{(23)}+4\gamma^{(12)}+4\gamma^{(13)}+4\gamma^{(23)}+12}$
\end{tabular}
\caption{$\kappa$ indices for the model $p$-$qI_n$ with $n=2$ (left) and $n=5$ (right).}\label{esempio1pqi}
\end{figure}

In Fig.~\ref{esempio1pqi} the $p$-$qI_n$ models are compared in the case of equal uniform marginals. From the plots, we observe that the pairwise $\kappa$'s increase faster in the case $n=2$ than in the case $n=5$, and there is strong correlation between the pairwise $\kappa$'s when they become near to $1$. In Fig.~\ref{esempio1pmix} we present plots for the same scenario, but for the $p$-$mix_n$ model. As it is immediate from the formulae, we obtain the same plot, meaning that function mapping the $\alpha$'s into the pairwise $\kappa$'s is independent on the number of categories.

\begin{figure}
\begin{tabular}{cc}
\includegraphics[width=6cm]{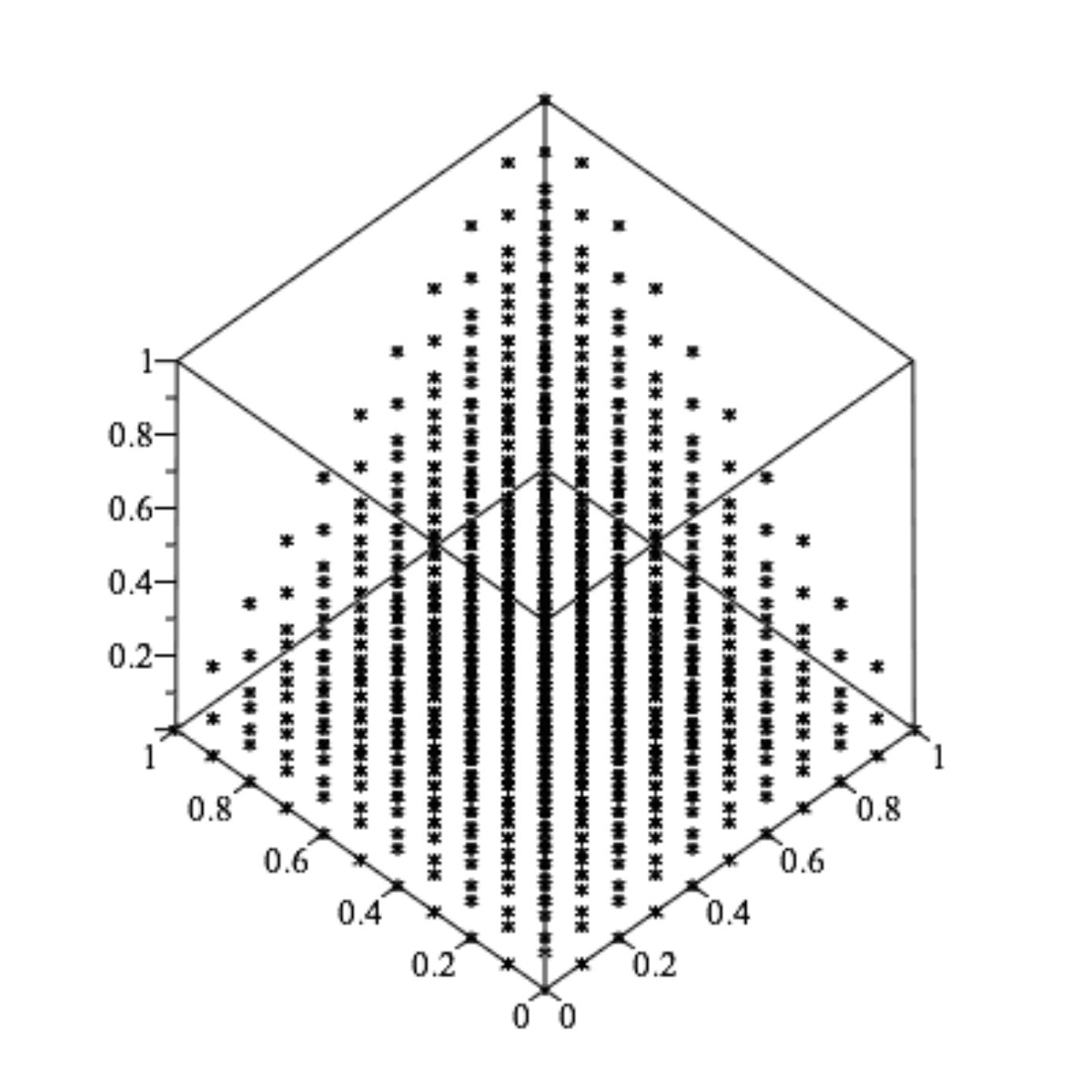} & \includegraphics[width=6cm]{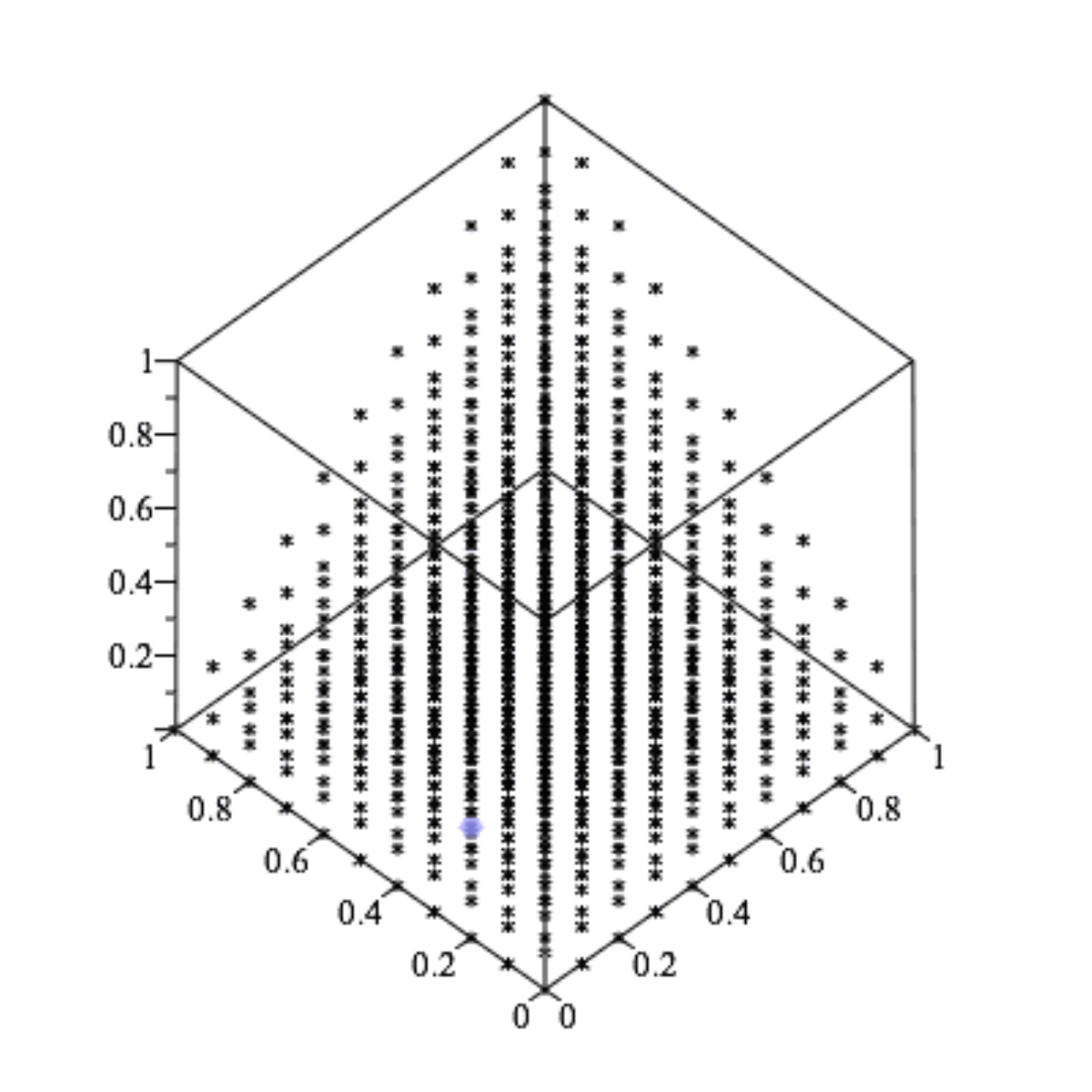}
\\
${\mathbf a}={\mathbf b}={\mathbf c}=\left( \frac12, \frac12\right)$ & ${\mathbf a}={\mathbf b}={\mathbf c}=\left( \frac15, \frac15, \frac15, \frac15, \frac15\right)$\\
\\
$\kappa_{12}=\alpha^{(12)}+\alpha^{(123)}$ &
$\kappa_{12}=\alpha^{(12)}+\alpha^{(123)}$ \\
\end{tabular}
\caption{$\kappa$ indices for the model $p$-$mix_n$ with $n=2$ (left) and $n=5$ (right).}\label{esempio1pmix}
\end{figure}

Finally, in Fig.~\ref{esempio2} we compare the plots for the $p$-$qI_n$ and the $p$-$mix_n$ models for $n=3$ and with different marginals. On the left side, there are the plots for equal and uniform marginals, while on the right side there are the plots for three different marginals. From the plots we can observe that both models are sensitive to the marginals. In the $p$-$qI_n$ model, asymmetric marginals yields more points with low values of pairwise $\kappa$'s, while in the  $p$-$mix_n$ model we have an opposite behavior.

\begin{figure}
\begin{tabular}{cc}
\includegraphics[width=6cm]{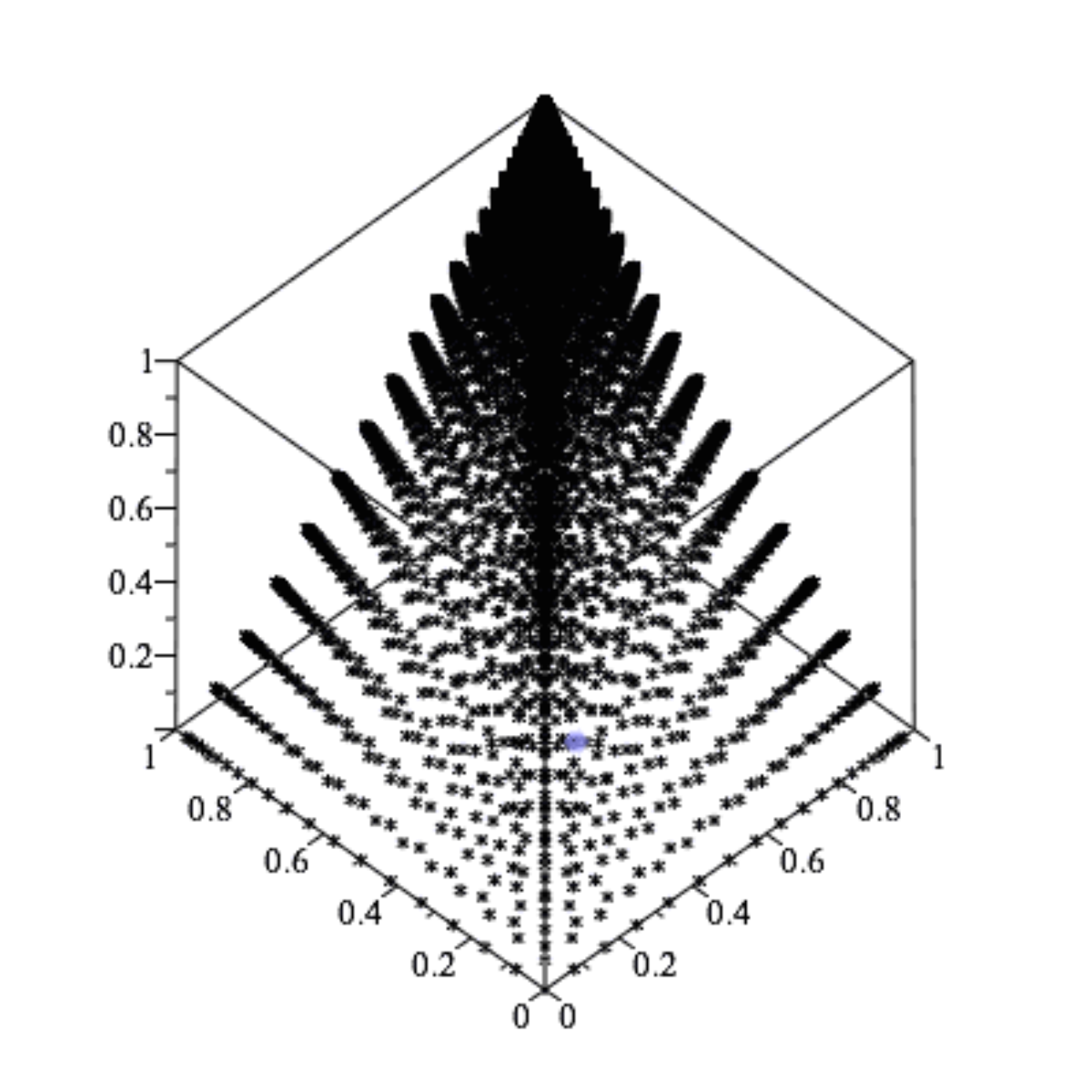} &\includegraphics[width=6cm]{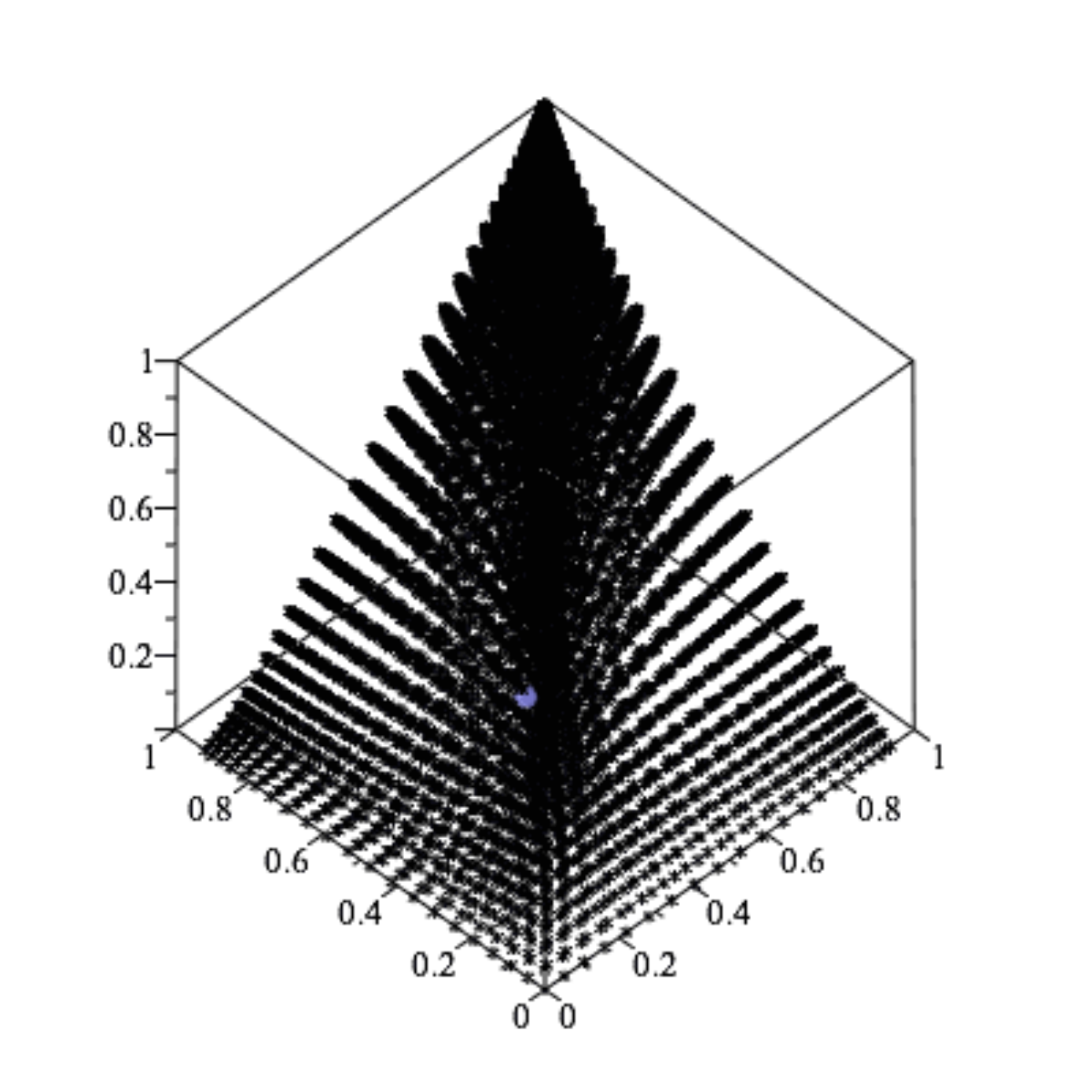}
\\
\includegraphics[width=6cm]{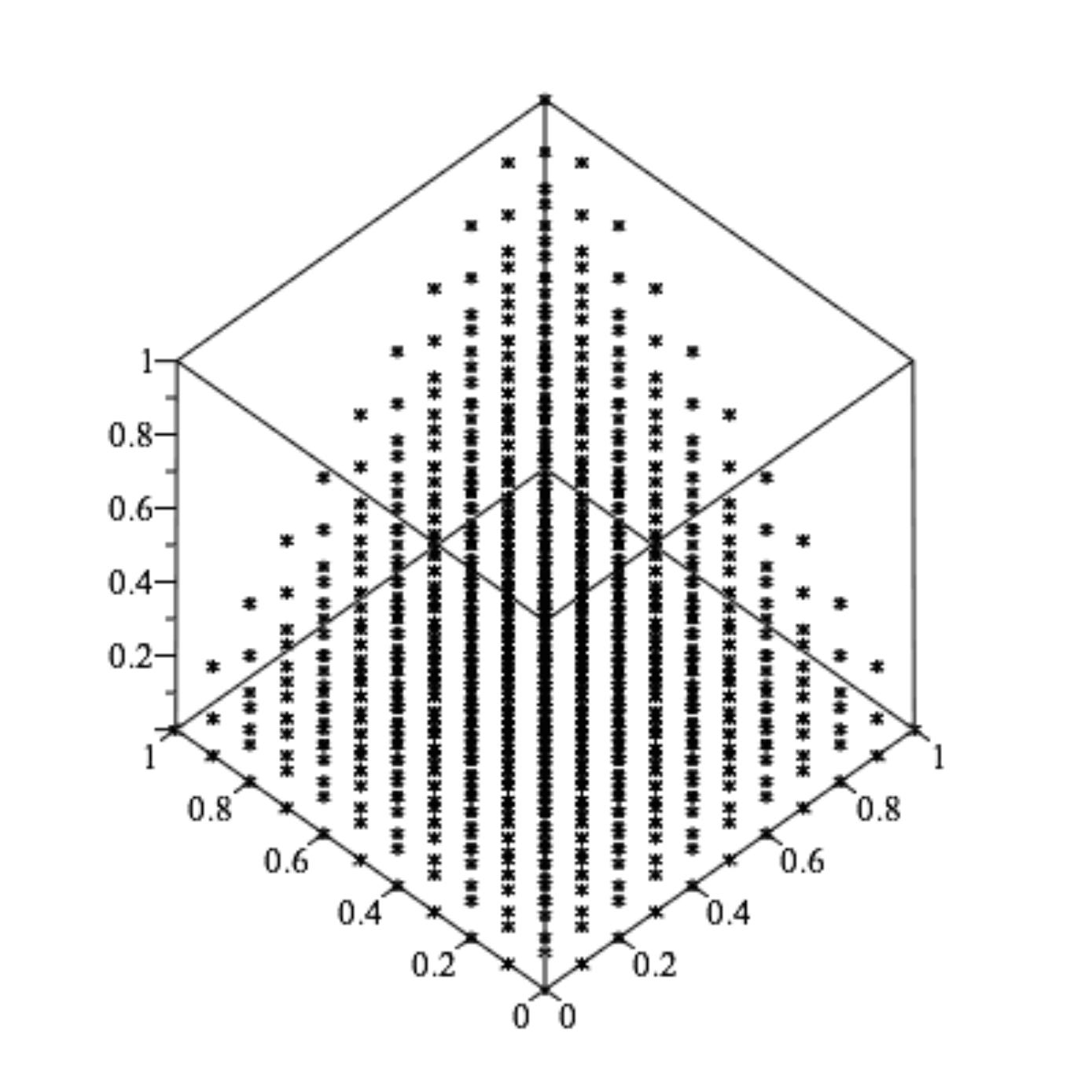} &\includegraphics[width=6cm]{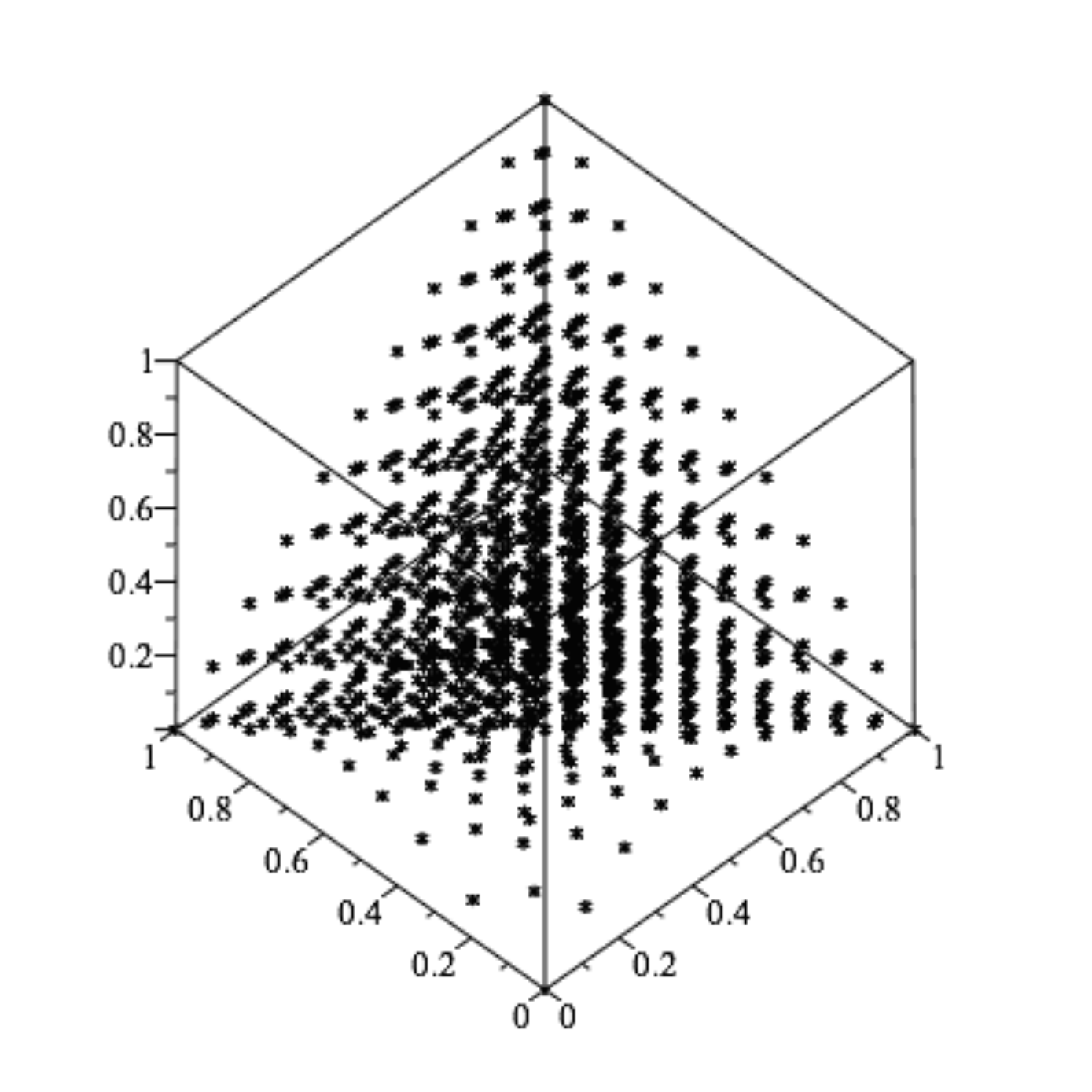} \\
${\mathbf a}={\mathbf b}={\mathbf c}=\left( \frac13, \frac13,\frac13 \right)$&
${\mathbf a}=\left( \frac{1}{10}, \frac{1}{10}, \frac{4}{5}\right)$\\
\\
& ${\mathbf b}=\left( \frac{1}{15}, \frac{2}{15}, \frac{4}{5}\right)$\\
\\
& ${\mathbf c}=\left( \frac{1}{6}, \frac{1}{6}, \frac{2}{3}\right)$
\end{tabular}
\caption{$\kappa$ indices for the model $p$-$qI_3$ (top) and the model $p$-$mix_3$ (bottom) with different values for ${\mathbf a}$, ${\mathbf b}$ and ${\mathbf c}$.}\label{esempio2}
\end{figure}

\section{Discussion}\label{discussion}

The work presented in this paper aims at providing a more thorough algebraic and geometric understanding of some models for the analysis of rater agreement. We have defined a class of toric models, with nice algebraic properties, and a class of mixture models, where the interpretation of the model parameters in terms of the pairwise Cohen's $\kappa$ is easy.

Under the point of view of Algebraic Statistics, the use of Cohen's $\kappa$ has been discussed in \cite{rapallo:05}, but only in the case of hypothesis testing, i.e., to test a null $\kappa$ versus a non-null $\kappa$. This test is particularly simple also within Algebraic Statistics, since the null hypothesis consists in the complete independence model. The inference under the complete independence model was one of the fist examples in Algebraic Statistics, and the corresponding invariants (i.e., the Markov basis) are fully characterized. When the problem is to measure the agrement among raters, the use of $\kappa$ is problematic and there is not unanimous consensus on the usefulness of kappa statistics to assess rater agreement. A better knowledge of the algebra and geometry of such index under the most common statistical models used in this field is a step towards a better use of the $\kappa$. While for the case of two raters some results can be found in \cite{schuster:01}, this paper extends the analysis to the case of three raters. Future work will include the extension and the analysis of the statistical models introduced here in the case of more than three raters.

\section*{Acknowledgements} We wish to thank the Institut Montpellierain Alexander Grothendieck, for the hospitality during the Workshop on Algebraic Statistics where this paper was written. The first author is a member of the INdAM-GNSAGA. The second author is a member of the INdAM-GNAMPA.

\bibliographystyle{plain}
\bibliography{biblio_BR}

\section*{Appendix A. Procedures in Maple${}^\mathrm{TM}$}

The following Maple${}^\mathrm{TM}$ procedures \texttt{pqI} and \texttt{pmix} compute the $\kappa$ indices for the models $p$-$qI_n$ and $p$-$mix_n$ respectively and plots a certain number (defined by the user) of points $(\kappa_{12},\kappa_{13},\kappa_{23})$.
The procedures take in input the size $n$ of the tensor, the vectors ${\mathbf a}$, ${\mathbf b}$, ${\mathbf c}$ and the number \texttt{step}, and return a matrix whose $(i,j)$-entry is $\kappa_{ij}$.
The number \texttt{step} behaves in a different way for the two procedures.
For $p$-$qI_n$, the procedure \texttt{pqI} plots the points $(\kappa_{12},\kappa_{13},\kappa_{23})$ as $\gamma^{(12)}$, $\gamma^{(13)}$ and $\gamma^{(23)}$ take values $10^0, 10^{0.1}, 10^{0.2}, \dots, 10^{step}$.
On the other hand, for $p$-$mix_n$,  the procedure \texttt{pmix} plots the points $(\kappa_{12},\kappa_{13},\kappa_{23})$ as $\alpha^{(0)}, \alpha^{(12)}, \alpha^{(13)}, \alpha^{(23)}, \alpha^{(123)}$ take values $0, 0+step, 0+2step, \dots, 1$ and sum to $1$. Hence, in this case the value in input for \texttt{step} must be of the form $\frac{1}{t}$, for $t$ a positive integer.

{\small{
\begin{verbatim}
with(plottools); with(plots);
pqI := proc (n, A, B, C, step) local i, j, k, p, psum, V1, V2, V3,
       T1, T2, T2l, T2r, kgen, k12, k13, k23, KM, counter, t, L1, L2, L3, G;

#creation of the parameterization
for i to n do
   for j to n do
      for k to n do
         if i = j and i <> k then
            p[i, j, k] := gamma[12]*A[i]*B[j]*C[k];
         end if;
         if i = k and i <> j then
            p[i, j, k] := gamma[13]*A[i]*B[j]*C[k];
         end if;
         if j = k and i <> j then
            p[i, j, k] := gamma[23]*A[i]*B[j]*C[k];
         end if;
         if i = j and i = k then
            p[i, j, k] := gamma[12]*gamma[13]*gamma[23]*A[i]*B[j]*C[k];
         end if;
         if i <> j and i <> k and j <> k then
            p[i, j, k] := A[i]*B[j]*C[k];
         end if;
      end do;
   end do;
end do;

psum := 0;
for i to n do
   for j to n do
      for k to n do
         psum := psum+p[i, j, k];
      end do;
   end do;
end do;

#definition of marginalization terms
V1 := Matrix(n);
V2 := Matrix(n);
V3 := Matrix(n);
for i to n do
   for j to n do
      for k to n do
         V1[i, j] := V1[i, j]+p[k, i, j];
         V2[i, j] := V2[i, j]+p[i, k, j];
         V3[i, j] := V3[i, j]+p[i, j, k];
      end do;
   end do;
end do;

#definition of  terms of kappa
T1 := 0;
for i to n do
   T1 := T1+tp[i, i];
end do;
T2 := 0;
for i to n do
   T2l := 0;
   T2r := 0;
   for j to n do
      T2l := T2l+tp[i, j];
      T2r := T2r+tp[j, i];
   end do;
   T2 := T2+T2l*T2r;
end do;

kgen := (psum*T1-T2)/(psum^2-T2);

#marginalization over first index
k23 := simplify(subs([seq(seq(tp[i, j] = V1[i, j], i = 1 .. n), j = 1 .. n)],
     kgen));

#marginalization over second index
k13 := simplify(subs([seq(seq(tp[i, j] = V2[i, j], i = 1 .. n), j = 1 .. n)],
    kgen));

#marginalization over third index
k12 := simplify(subs([seq(seq(tp[i, j] = V3[i, j], i = 1 .. n), j = 1 .. n)],
    kgen));

KM := Matrix([[0, k12, k13], [k12, 0, k23], [k13, k23, 0]]);

#creating plot
counter := 1;
for t[12] from 0 by .1 to step do
   for t[13] from 0 by .1 to step do
      for t[23] from 0 by .1 to step do
         L1 := subs([gamma[12] = 10^t[12], gamma[13] = 10^t[13],
                            gamma[23] = 10^t[23]], k12);
         L2 := subs([gamma[12] = 10^t[12], gamma[13] = 10^t[13],
                            gamma[23] = 10^t[23]], k13);
         L3 := subs([gamma[12] = 10^t[12], gamma[13] = 10^t[13],
                            gamma[23] = 10^t[23]], k23);
         G[counter] := plottools:-point([L1, L2, L3], color = black,
                             symbol = cross, symbolsize = 15);
         counter := counter+1;
      end do;
   end do;
end do;

#printing plot
print(plots:-display(seq(G[i], i = 1 .. counter-1), axes = boxed,
                  orientation = [45, 135], view = [0 .. 1, 0 .. 1, 0 .. 1]));

#return matrix
return (KM);
end proc;
\end{verbatim}
}}
\vskip0.2cm
{\small{
\begin{verbatim}
pmix := proc (n, A, B, C, step) local i, j, k, p, psum, V1, V2, V3,
        T1, T2, T2l, T2r, kgen, k12, k13, k23, KM, counter, t, L1, L2, L3, G;

#creation of the parameterization
for i to n do
   for j to n do
      for k to n do
         if i = j and i <> k then
            p[i, j, k] := alpha[0]*A[i]*B[j]*C[k]+alpha[12]/n^2;
         end if;
         if i = k and i <> j then
            p[i, j, k] := alpha[0]*A[i]*B[j]*C[k]+alpha[13]/n^2;
         end if;
         if j = k and i <> j then
            p[i, j, k] := alpha[0]*A[i]*B[j]*C[k]+alpha[23]/n^2;
         end if;
         if i = j and i = k then
            p[i, j, k] := alpha[0]*A[i]*B[j]*C[k]+alpha[12]/n^2
                 +alpha[13]/n^2+alpha[23]/n^2+alpha[123]/n;
          end if;
          if i <> j and i <> k and j <> k then
               p[i, j, k] := alpha[0]*A[i]*B[j]*C[k];
          end if;
      end do;
   end do;
end do;
psum := 0;
for i to n do
   for j to n do
      for k to n do
         psum := psum+p[i, j, k];
      end do;
   end do;
end do;

#definition of marginalization terms
V1 := Matrix(n);
V2 := Matrix(n);
V3 := Matrix(n);
for i to n do
   for j to n do
      for k to n do
         V1[i, j] := V1[i, j]+p[k, i, j];
         V2[i, j] := V2[i, j]+p[i, k, j];
         V3[i, j] := V3[i, j]+p[i, j, k];
      end do;
   end do;
end do;

#definition of  terms of kappa
T1 := 0;
for i to n do
   T1 := T1+tp[i, i];
end do;
T2 := 0;
for i to n do
   T2l := 0;
   T2r := 0;
   for j to n do
      T2l := T2l+tp[i, j];
      T2r := T2r+tp[j, i];
   end do;
   T2 := T2+T2l*T2r;
end do;
kgen := (psum*T1-T2)/(psum^2-T2);

#marginalization over first index
k23 := simplify(subs([seq(seq(tp[i, j] = V1[i, j], i = 1 .. n), j = 1 .. n)],
     kgen));

#marginalization over second index
k13 := simplify(subs([seq(seq(tp[i, j] = V2[i, j], i = 1 .. n), j = 1 .. n)],
    kgen));

#marginalization over third index
k12 := simplify(subs([seq(seq(tp[i, j] = V3[i, j], i = 1 .. n), j = 1 .. n)],
    kgen));

KM := Matrix([[0, k12, k13], [k12, 0, k23], [k13, k23, 0]]);

#creating plot
counter := 1;
for t[0] from 0 by step to 1 do
   for t[1] from 0 by step to 1-t[0] do
      for t[2] from 0 by step to 1-t[0]-t[1] do
         for t[3] from 0 by step to 1-t[0]-t[1]-t[2] do
            L1 := subs([alpha[0] = t[0], alpha[12] = t[1], alpha[13] = t[2],
                  alpha[23] = t[3], alpha[123] = 1-t[0]-t[1]-t[2]-t[3]], k12);
            L2 := subs([alpha[0] = t[0], alpha[12] = t[1], alpha[13] = t[2],
                  alpha[23] = t[3], alpha[123] = 1-t[0]-t[1]-t[2]-t[3]], k13);
            L3 := subs([alpha[0] = t[0], alpha[12] = t[1], alpha[13] = t[2],
                  alpha[23] = t[3], alpha[123] = 1-t[0]-t[1]-t[2]-t[3]], k23);
            G[counter] := plottools:-point([L1, L2, L3], color = black,
                                symbol = cross, symbolsize = 15);
            counter := counter+1;
         end do;
      end do;
   end do;
end do;

#printing plot
print(plots:-display(seq(G[i], i = 1 .. counter-1), axes = boxed,
                  orientation = [45, 135], view = [0 .. 1, 0 .. 1, 0 .. 1]));

#return matrix
return (KM);
end proc;
\end{verbatim}
}}

\end{document}